\documentclass[10pt]{amsart}

\usepackage{amssymb,amsthm,amsmath}
\usepackage[numbers,sort&compress]{natbib}
\usepackage{color}
\usepackage{graphicx}
\usepackage{tikz}

\title[ ]{KAM theorem with large twist and finite smooth large  perturbation}

\thanks{
This work is supported by the Doctoral Starting Foundation of Quzhou University (No. BSYJ202115).}

\author{Lu Chen}
\address[Lu Chen]{College of Teacher Education, Quzhou University, Quzhou, Zhejiang 324000,
P. R. China} \email{chenlu@zju.edu.cn}

\theoremstyle{plain}
\newtheorem{thm}{Theorem}[section]
 
 \newtheorem{lem}[thm]{Lemma}
 
 \theoremstyle{definition}
 
 \theoremstyle{remark}
 \newtheorem{rem}[thm]{Remark}
 
 \numberwithin{equation}{section}
\begin{document}

\begin{abstract}
In the present paper, we will discuss the following non-degenerate Hamiltonian system
\begin{equation*}
H(\theta,t,I)=\frac{H_0(I)}{\varepsilon^{a}}+\frac{P(\theta,t,I)}{\varepsilon^{b}},
\end{equation*}
 where $(\theta,t,I)\in\mathbf{{T}}^{d+1}\times[1,2]^d$ ($\mathbf{{T}}:=\mathbf{{R}}/{2\pi \mathbf{Z}}$), $a,b$ are given positive constants with $a>b$, $H_0: [1,2]^d\rightarrow \mathbf R$ is real analytic and $P: \mathbf T^{d+1}\times [1,2]^d\rightarrow \mathbf R$ is $C^{\ell}$ with $\ell=\frac{2(d+1)(5a-b+2ad)}{a-b}+\mu$, $0<\mu\ll1$. We prove that if $\varepsilon$ is sufficiently small,  there is an invariant torus with given Diophantine frequency vector for the above Hamiltonian system. As for application, we prove that a finite network of Duffing oscillators with periodic exterior forces possesses Lagrangian stability for almost all initial data.

\end{abstract}

\maketitle
\section{ Introduction and main results}\label{sec1}
Consider the harmonic oscillator (linear spring)
\begin{equation}\label{fc1-1}
\ddot{x}+k^2x=0.
\end{equation}
It is well-known that any solution of this equation is periodic. So any solution of this equation is bounded for $t\in \mathbf{R}$. That is, this equation is Lagrange stable. However,  there is an unbounded solution to the equation
\begin{equation}\label{fc1-2}
\ddot{x}+k^2\, x=p(t)
\end{equation}
where the frequency of $p$ is equal to the frequency $k$ of the spring itself.  Now let us consider a nonlinear equation
\begin{equation}\label{fc1-3}
\ddot{x}+x^3=0.
\end{equation}
 This equation is Lagrange stable, too. An interesting problem is that, does
\begin{equation}\label{fc1-4}
\ddot{x}+x^3=p(t)
\end{equation}
have Lagrange stability when $p(t)$ is periodic?
    Moser \cite{a1, a2} proposed to study the boundedness of all solutions for Duffing equation
\begin{equation}\label{fc1-5}
 \ddot{x}+\alpha x^3+\beta x=p(t),
 \end{equation}
where $\alpha>0, \beta\in \mathbf{R}$ are constants, $p(t)$ is a $1$-periodic continuous function.
The first boundedness result, prompted by questions of Littlewood \cite{a3}, is due to Morris \cite{a4} in 1976 who showed that all solutions of the equation (\ref{fcb1-5}) are bounded for all time.
\begin{equation}\label{fcb1-5}
 \ddot{x}+2x^3=p(t),
 \end{equation}
where $p(t)$ is a $2\pi$-periodic continuous function. Subsequently, Morris's boundedness result was, by Dieckerhoff-Zehnder \cite{a5} in 1987, extended to a wider class of systems
\begin{equation}\label{fcb1-6}
\ddot{x}+x^{2n+1}+\sum_{i=0}^{2n} x^{i}p_{i}(t)=0, n\geq 1,
\end{equation}
where $p_i(t)\in C^{\infty} (i=0,1,\cdots,2n)$  are 1-periodic functions.
For some other extensions to study the boundedness, one may see papers \cite{a6, a7, a8, a9, a10, a11, a12, a13, a14, a15}.

In many research fields such as physics, mechanics and mathematical biology as so on arise networks of coupled Duffing oscillators of various form.
 For example, the evolution equations for
the voltage variables $V_1$ and $V_2$ obtained using the Kirchhoff's
voltage law are
\begin{equation}\label{fc1-6}
\begin{cases}
  R^2 C^2 \frac{d^2 V_1}{d t^2}=-(\frac{R^2 C}{R_1})\frac{d V_1}{dt}-(\frac{R}{R_2})V_1-(\frac{R}{100 R_3}) V_1^3+(\frac{R}{R_C})V_2+f \sin \omega t,\\
    R^2 C^2 \frac{d^2 V_2}{d t^2}=-(\frac{R^2 C}{R_1})\frac{d V_2}{dt}-(\frac{R}{R_2})V_2-(\frac{R}{100 R_3}) V_2^3+(\frac{R}{R_C})V_1 , \end{cases}
\end{equation}
 where $R$'s and $C$'s are resistors and capacitors, respectively.  This equation can be regarded as one  coupled by two Duffing oscillators. See  \cite{a17, a18, a19, a20, a21, a22, a23, a24} for more details.

Recently, Yuan-Chen-Li \cite{a16} studied the Lagrangian stability for coupled Hamiltonian system of $m$ Duffing oscillators:
\begin{equation}\label{fc1-7}
\ddot{x_{i}}+x_{i}^{2n+1}+\frac{\partial F}{\partial x_{i}}=0,\ \ i=1, 2, \cdots, m,
\end{equation}
where the polynomial potential $F=F(x, t)=\sum_{\alpha\in \mathbf{N}^{m}, |\alpha|\leq 2n+1}p_{\alpha}(t)x^{\alpha},$ $x\in \mathbf{R}^{m}$ with $p_{\alpha}(t)$ is of period $2 \pi$,  and $n$ is a given natural number. Yuan-Chen-Li \cite{a16} proved that (\ref{fc1-7}) had Lagrangian stability for almost all initial data if $p_{\alpha}(t)$ was real analytic.

In the present paper, we will relax the real analytic condition of $p_{\alpha}(t)$ to $C^{\ell}$ ($\ell=2(m+1)(4n+2nm+1)+\mu$ with $0<\mu\ll1$).

In the whole of the present paper we denote by $C$ (or $C_0, C_1, c, c_0,c_1$, etc) an universal constant which may be different in different places.  Let positive integer $d$ be the freedom of the to-be considered Hamiltonian.

\begin{thm}\label{thm1-1} Consider a Hamiltonian \begin{equation}\label{fcb1-1}
H(\theta,t,I)=\frac{H_0(I)}{\varepsilon^{a}}+\frac{P(\theta,t,I)}{\varepsilon^{b}},
\end{equation}
 where $a,b$ are given positive constants with $a>b$, and $H_0$ and $P$ obey the following conditions:\\
{\rm(1)} Given $\ell=\frac{2(d+1)(5a-b+2ad)}{a-b}+\mu$ with $0<\mu\ll1$, and $H_0: [1,2]^d\rightarrow \mathbf R$ is real analytic and $P: \mathbf T^{d+1}\times [1,2]^d\rightarrow \mathbf R$ is $C^{\ell}$, and
\begin{equation}\label{fcb1-2}
||H_0||:=\sup_{I\in[1,2]^d}|H_0(I)|\le c_1, \ |P|_{C^{\ell}(\mathbf T^{d+1}\times [1,2]^d)}\le c_2,
\end{equation}
{\rm(2)} $H_0$ is non-degenerate in Kolmogorov's sense:
\begin{equation}\label{fcb1-3}
\text{det}\,\left(\frac{\partial^2 H_0(I)}{\partial I^2}\right)\ge c_3>0,\forall\; I\in [1,2]^d.
\end{equation}
Then there exists $0<\epsilon^*\ll 1$ such that for any $\varepsilon$ with $0<\varepsilon<\epsilon^*$,  the Hamiltonian system
$$\dot \theta=\frac{\partial H(\theta,t,I)}{\partial I},\;\dot I=-\frac{\partial H(\theta,t,I)}{\partial \theta}$$ possesses a $d+1$ dimensional invariant torus of rotational frequency vector $(\omega(I_0),2\pi)$ with $\omega(I):=\frac{\partial H_0(I)}{\partial I}$, for any $I_0\in  [1,2]^d$  and $\omega(I_0)$ obeying   Diophantine conditions {\rm(we let $B=5a-b+2ad$):}\\
{\rm(i)}
\begin{equation}\label{fc1-19}
|\frac{\langle k,\omega(I_0)\rangle}{\varepsilon^a} +l|\ge \frac{\varepsilon^{-a+\frac{B}{\ell}}\gamma}{|k|^{\tau_1}}>\frac{\gamma}{|k|^{\tau_2}},\ \ k\in\mathbf Z^d\setminus\{0\},l\in\mathbf Z, |k|+|l|\le \varepsilon^{-\frac{B}{\ell}}(\log\frac{1}{\varepsilon})^2,
\end{equation}
where $\gamma=(\log\frac{1}{\varepsilon})^{-4},$ $\tau_1=d-1+\frac{(a-b)^2\mu}{1000(a+b+1)(d+3)(5a-b+2ad)}$, $\tau_2=d+\frac{(a-b)^2\mu}{1000(a+b+1)(d+3)(5a-b+2ad)}$; \\
{\rm(ii)}
\begin{equation}\label{fc1-20}
|\frac{\langle k,\omega(I_0)\rangle}{\varepsilon^a} +l|\ge \frac{\gamma}{|k|^{\tau_2}},\ \ k\in\mathbf Z^d\setminus\{0\},l\in\mathbf Z, |k|+|l|> \varepsilon^{-\frac{B}{\ell}}(\log\frac{1}{\varepsilon})^{2},
  \end{equation}
where $\gamma=(\log\frac{1}{\varepsilon})^{-4},$ $\tau_2=d+\frac{(a-b)^2\mu}{1000(a+b+1)(d+3)(5a-b+2ad)}$.
\end{thm}
Applying Theorem \ref{thm1-1} to (\ref{fc1-7}) we have the following theorem.

\begin{thm}\label{thm1-2} For any $A>0$, let $\Theta_A=\{(x_{1}, \dot x_{1}; \cdots, x_{m}, \dot x_{m})\in\mathbf R^{2m}:\ A\le \sum_{i=1}^m x_i^{2n+2}+(n+1)\dot x^2_i\le c_4A, \ c_4>1\}$. Then there exists a subset $\tilde \Theta_A\subset\Theta_A$ with
\begin{equation}\label{fc1-21}
\lim_{A\to\infty}\frac{\tilde \Theta_A}{\Theta_A}=1
 \end{equation}
such that any solution to equation {\rm(\ref{fc1-7})} with any initial data $(x_{1}(0), \dot x_{1}(0); \cdots, x_{m}(0), \dot x_{m}(0))\in \tilde\Theta_A $ is time quasi-periodic with frequency vector $(\omega,2\pi)$ where $\omega=(\omega_i:\ i=1,\cdots,m)$ and $\omega_i=\omega_i(I(0))$ with  $I(0)=(I_1(0),\cdots,I_m(0))$, $I_{i}(0)=(n+1)\dot x_{i}^{2}(0)+x_{i}^{2n+2}(0)$, furthermore,
\begin{equation}\label{fc1-22}
\sup_{t\in\mathbf R}\sum_{i=1}^m |x_i(t)|+|\dot x_i(t)|<\infty.
\end{equation}

\end{thm}

\begin{rem}\label{rem1-3} An equation is called to have Lagrangian stability for almost all initial data if its solutions obey (\ref{fc1-21}) and (\ref{fc1-22}).
\end{rem}

\begin{rem}\label{rem1-4} Let $\Theta=\{I_0\in[1,2]^d:\ \omega(I_0) \,\text{obeys the Diophantine conditions}\}$. We claim that the Lebesgue measure of $\Theta$ approaches to $1$:
\[\text{Leb} \Theta\ge 1-C (\log\frac{1}{\varepsilon})^{-2}\to 1,\;\text{as}\; \varepsilon
\to 0.\]
Let
$$\tilde\Theta_{k,l}=\left\{\xi\in\omega([1,2]^d):\ |\frac{\langle k,\xi\rangle}{\varepsilon^a} +l|\le \frac{\varepsilon^{-a+\frac{B}{\ell}}\gamma}{|k|^{\tau_1}},\ k\in\mathbf Z^d\setminus\{0\},l\in\mathbf Z, |k|+|l|\le \varepsilon^{-\frac{B}{\ell}}(\log\frac{1}{\varepsilon})^2\right\}$$
 and
$$\tilde\Theta_{k,l}=\left\{\xi\in\omega([1,2]^d):\ |\frac{\langle k,\xi\rangle}{\varepsilon^a} +l|\le \frac{\gamma}{|k|^{\tau_2}},\ k\in\mathbf Z^d\setminus\{0\},l\in\mathbf Z, |k|+|l|> \varepsilon^{-\frac{B}{\ell}}(\log\frac{1}{\varepsilon})^{2}\right\}.$$
Let $f(\xi)=\frac{\langle k, \xi\rangle}{\varepsilon^{a}}+l.$ Since $k\neq 0,$ there exists an unit vector $v\in\mathbf{Z}^{d}$ such that
\begin{equation}\label{fc1-23}
\frac{df(\xi)}{dv}\geq \frac{C|k|}{\varepsilon^{a}}.
\end{equation}

Then, if , $k\in\mathbf Z^d\setminus\{0\},l\in\mathbf Z, |k|+|l|\le \varepsilon^{-\frac{B}{\ell}}(\log\frac{1}{\varepsilon})^{2}$, by (\ref{fc1-23}), we have
\begin{equation}\label{fc1-24}
\text{Leb}\tilde\Theta_{k,l}\le C \frac{\gamma\cdot\varepsilon^{\frac{B}{\ell}}}{|k|^{\tau_1+1}}.
\end{equation}
Thus,
\begin{equation}\label{fc1-25}
\text{Leb}\ \left(\bigcup_{k\in\mathbf Z^d\setminus\{0\},l\in\mathbf Z,|k|+|l|\le \varepsilon^{-\frac{B}{\ell}}(\log\frac{1}{\varepsilon})^{2} }\tilde\Theta_{k,l}\right)\le \sum_{|l|\le \varepsilon^{-\frac{B}{\ell}}(\log\frac{1}{\varepsilon})^{2}}C \gamma\cdot\varepsilon^{\frac{B}{\ell}}\le C(\log\frac{1}{\varepsilon})^{-2}.
\end{equation}

If $k\in\mathbf Z^d\setminus\{0\}, \ l\in\mathbf Z, \ |k|+|l|> \varepsilon^{-\frac{B}{\ell}}(\log\frac{1}{\varepsilon})^{2}$, we can let $c_5=\max\{|\omega([1,2]^d)|\}:=\max\{\sum_{i=1}^d|\omega_i([1,2]^d)|\}.$ Noting that $|\langle k, \xi\rangle|\leq c_5|k|.$ Thus if $|l|>\frac{c_5|k|}{\varepsilon^a}+1,$ then
$$|\frac{\langle k, \xi\rangle}{\varepsilon^a}+l|\geq|l|-|\frac{\langle k, \xi\rangle}{\varepsilon^a}|>\frac{c_5|k|}{\varepsilon^a}+1-\frac{c_5|k|}{\varepsilon^a}\geq1>\frac{\gamma}{|k|^{\tau_2}}.$$
It follows that $\tilde\Theta_{k,l}=\phi.$ Now we assume $|l|\leq\frac{c_5|k|}{\varepsilon^a}+1$, then by (\ref{fc1-23}),we have
\begin{equation}\label{fc1-26}
\text{Leb}\tilde\Theta_{k,l}\le \frac{C\gamma\varepsilon^a}{|k|^{\tau_2+1}}.
\end{equation}
Thus,
\begin{eqnarray}\label{fc1-27}
\nonumber\text{Leb}\ \left(\bigcup_{k\in\mathbf Z^d\setminus\{0\},l\in\mathbf Z,|k|+|l|> \varepsilon^{-\frac{B}{\ell}}(\log\frac{1}{\varepsilon})^{2} }\tilde\Theta_{k,l}\right)&\le& \sum_{k\neq0}\sum_{|l|\leq\frac{c_5|k|}{\varepsilon^a}+1}\frac{C\gamma\varepsilon^a}{|k|^{\tau_2+1}}\\
&\le&\sum_{k\neq0}\frac{C\gamma}{|k|^{\tau_2}}\le C(\log\frac{1}{\varepsilon})^{-4}.
\end{eqnarray}
Let $\Theta=[1,2]^d\setminus \left(\bigcup_{k\in\mathbf Z^d\setminus\{0\},l\in\mathbf Z}\omega^{-1}(\tilde\Theta_{k,l})\right)$. By the Kolmogorov's non-degenerate condition, the map $\omega:[1,2]^d\to \omega([1,2]^d)$ is a diffeomorphism in both direction. Then by (\ref{fc1-25}) and (\ref{fc1-27}), the proof of the claim is completed by letting  $\Theta=[1,2]^d\setminus \left(\bigcup_{k\in\mathbf Z^d\setminus\{0\},l\in\mathbf Z}\omega^{-1}(\tilde\Theta_{k,l})\right)$.
\end{rem}

\section{Approximation Lemma}\label{sec2}
First we denote by $|\cdot|$ the norm of any finite dimensional Euclidean space. Let $C^{\tilde{\mu}}(\mathbf{R}^{m})$ for $0<\tilde{\mu}<1$ denote the space of bounded H\"older continuous functions $f: \mathbf{R}^m\rightarrow\mathbf{R}^n$ with the norm
\begin{equation*}
|f|_{C^{\tilde{\mu}}}=\sup_{0<|x-y|<1}\frac{|f(x)-f(y)|}{|x-y|^{\tilde{\mu}}}+\sup_{x\in\mathbf{R}^m}|f(x)|.
\end{equation*}
If $\tilde{\mu}=0$ the $|f|_{C^{\tilde{\mu}}}$ denotes the sup-norm. For $\tilde{\ell}=k+\tilde{\mu}$ with $k\in\mathbf{N}$ and $0\leq\tilde{\mu}<1$  we denote
by $C^{\tilde{\ell}}(R^{m})$ the space of functions $f:\mathbf{R}^m\rightarrow \mathbf{R}^n$ with H\"older continuous partial derivatives  $\partial^{\alpha}f\in C^{\tilde{\mu}}(\mathbf{R}^m)$ for all multi-indices $\alpha=(\alpha_1,\cdots,\alpha_m)\in\mathbf{N}^m$ with the assumption that $|\alpha|:=|\alpha_1|+\cdots+|\alpha_m|\leq k$. We define the norm
\begin{equation*}
  |f|_{C^{\tilde{\ell}}}:=\sum_{|\alpha|\leq\tilde{\ell}}|\partial^{\alpha}f|_{C^{\tilde{\mu}}}
\end{equation*}
for $\tilde{\mu}=\tilde{\ell}-[\tilde{\ell}]<1$. In order to give an approximate lemma, we define the kernel function
\begin{equation*}
  K(x)=\frac{1}{(2\pi)^m}\int_{\mathbf{R}^m}\hat{K}(\xi)e^{i\langle x,\xi\rangle}d\xi,\ x\in\mathbf{C}^m,
\end{equation*}
where $\hat{K}(\xi)$ is a $C^{\infty}$ function with compact support, contained in the ball $|\xi|\leq a_1$ with a constant $a_1>0$, that satisfies
$$ \partial^{\alpha}\hat{K}(0)=
\left\{
\begin{array}{l}
1, \ \alpha=0,\\
0, \ \alpha\neq0.
\end{array}
\right.
$$
Then $K: \mathbf{C}^m\rightarrow\mathbf{R}^n$ is a real analytic function with the property that for every $j>0$ and every $p>0$, there exists a constant $c=c(j,p)>0$ such that for all $\beta\in \mathbf{N}^m$ with $|\beta|\leq j$,
\begin{equation}\label{fc2-1}
  |\partial^{\beta}K(x+iy)|\leq c(1+|x|)^{-p}e^{a_1|y|}, \ x,y\in \mathbf{R}^m.
\end{equation}
\begin{lem}[Jackson-Moser-Zehnder]\label{lem2-1}
There is a family of convolution operators
\begin{equation}\label{fc2-2}
 (S_{s}F)(x)=s^{-m}\int_{\mathbf{R}^{m}}K(s^{-1}(x-y))F(y)dy,\ \ 0<s\leq 1,\ \ \forall\ F\in C^{0}(\mathbf{R}^{m})
 \end{equation}
 from $C^{0}(\mathbf{R}^{m})$ into the linear space of entire functions on $\mathbf{C}^{m}$ such that for
every $\ell>0$ there exist a constant $c=c(\tilde{\ell})>0$ with the following properties: if $F\in C^{\tilde{\ell}}(\mathbf{R}^{m}),$ then for $|\alpha|\leq \tilde{\ell}$ and $|\mathrm{Im} x|\leq s$,
\begin{equation}\label{fc2-3}
|\partial^{\alpha}(S_{s}F)(x)-\sum_{|\beta|\leq \tilde{\ell}-|\alpha|}\partial^{\alpha+\beta}F(\mathrm{Re} x)({\bf{i}}\, \mathrm{Im} x)^{\beta}/\beta!|\leq c|F|_{C^{\tilde{\ell}}}s^{\tilde{\ell}-|\alpha|}
\end{equation}
and in particular for $\rho\leq s$
\begin{equation}\label{fc2-4}
|\partial^{\alpha}S_{s}F-\partial^{\alpha}S_{\rho}F|_{\rho}:=\sup_{|{\rm Im} x|\leq \rho}|\partial^{\alpha}(S_{s}F)(x)-\partial^{\alpha}(S_{\rho}F)(x)|\leq c|F|_{C^{\tilde{\ell}}}s^{\tilde{\ell}-|\alpha|}.
\end{equation}
Moreover, in the real case
\begin{equation}\label{fc2-5}
 |S_{s}F-F|_{C^{p}}\leq c |F|_{C^{\tilde{\ell}}}s^{\tilde{\ell}-p},\ \ p\leq \tilde{\ell},
\end{equation}
\begin{equation}\label{fc2-6}
 |S_{s}F|_{C^{p}}\leq c|F|_{C^{\tilde{\ell}}}s^{\tilde{\ell}-p},\ \ p\leq \tilde{\ell}.
\end{equation}
Finally, if $F$ is periodic in some variables then so are the approximating functions $S_{s}F$ in the same variables.
\end{lem}
\begin{rem}\label{rem2-1}
Moreover we point out that from (\ref{fc2-6}) one can easily deduce the following well-known convexity estimates
\begin{equation}\label{fc2-7}
|f|_{C^{q}}^{l-k}\leq c|f|_{C^{k}}^{l-q}|f|_{C^{l}}^{q-k},\ \ k\leq q\leq l,
\end{equation}
\begin{equation}\label{fc2-8}
|f\cdot g|_{C^{s}}\leq c(|f|_{C^{s}}|f|_{C^{0}}+|f|_{C^{0}}|g|_{C^{s}}),\ \ s\geq 0.
\end{equation}
See \cite{a25, a26} for the proofs of Lemma \ref{lem2-1} and the inequalities (\ref{fc2-7}) and (\ref{fc2-8}).
\end{rem}
\begin{rem}\label{rem2-2}
From the definition of the operator $S_s$, we clearly have
\begin{equation}\label{fc2-9}
\sup_{x,y\in\mathbf{R}^{m}, |y|\leq s}|S_sF(x+iy)|\leq C|F|_{C^0}.
\end{equation}
In fact, by the definition of $S_s$, we have that for any $x,y\in\mathbf{R}^{m}$ with $|y|\leq s$,
\begin{eqnarray}
 \nonumber|S_sF(x+iy)|&=&\nonumber |s^{-m}\int_{\mathbf{R}^{m}}K(s^{-1}(x+iy-z))F(z)dz|\\
&=&\nonumber |\int_{\mathbf{R}^{m}}K(is^{-1}y+\xi)F(x-s\xi)d\xi|\\
&\leq&\nonumber |F|_{C^0}\int_{\mathbf{R}^{m}}|K(is^{-1}y+\xi)|d\xi\\
&\leq&\nonumber C|F|_{C^0},
\end{eqnarray}
where we used (\ref{fc2-1}) in the last inequality.
\end{rem}

Consider a  function $F(\theta,t,I)$, where $F:\mathbf T^{d+1}\times [1,2]^d\rightarrow\mathbf{R}$ satisfies
\begin{equation*}
|F|_{C^{\tilde{\ell}}(\mathbf T^{d+1}\times [1,2]^d)}\leq C.
 \end{equation*}
 By Whitney's extension theorem, we can find a function $\tilde{F}: \mathbf{T}^{d+1}\times \mathbf{R}^{d}\rightarrow\mathbf{R}$ such that $\tilde{F}|_{\mathbf T^{d+1}\times [1,2]^d}=F$ (i.e. $\tilde{F}$ is the extension of $F$) and
\begin{equation*}
|\tilde{F}|_{C^{|\alpha|}(\mathbf{T}^{d+1}\times \mathbf{R}^{d})}\leq C_{\alpha}|F|_{C^{|\alpha|}(\mathbf T^{d+1}\times [1,2]^d)}, \ \forall\alpha\in\mathbf{Z}^{2d+1}_{+}, |\alpha|\leq\tilde{\ell},
 \end{equation*}
 where $C_{\alpha}$ is a constant depends only $\tilde{\ell}$ and $d$.

 Let $z=(\theta,t,I)$ for brevity, define, for $\forall s>0$,
 \begin{equation*}
 (S_{s}\tilde{F})(z)=s^{-(2d+1)}\int_{\mathbf{T}^{d+1}\times \mathbf{R}^{d}}K(s^{-1}(z-\tilde{z}))\tilde{F}(\tilde{z})d\tilde{z}.
 \end{equation*}
For any positive integer $p$, let $\mathbf{{T}}_s^p=\left\{x\in \mathbf{{C}}^{p}/{(2\pi \mathbf{Z})}^{p}: |{\rm Im}x|\leq s \right\}$, $\mathbf{{R}}_s^p=\left\{x\in \mathbf{{C}}^{p}: |{\rm Im}x|\leq s\right\}$.  Fix a sequence of fast decreasing
numbers $s_{\nu}\downarrow0$, $\nu\in\mathbf{Z_{+}}$ and $s_0\leq\frac{1}{4}$. Let
 \begin{equation*}
 F^{(\nu)}(z)=(S_{2s_\nu}\tilde{F})(z),\ \nu\geq0.
 \end{equation*}
Then $F^{(\nu)}$'s ($\nu\geq0$) are entire functions in $\mathbf{C}^{2d+1}$, in particular, which obey the following properties.\\
(1) $F^{(\nu)}$'s ($\nu\geq0$) are real analytic on the complex domain $\mathbf{T}^{d+1}_{2s_{\nu}}\times \mathbf{R}^{d}_{2s_{\nu}}$;\\
(2) The sequence of functions $F^{(\nu)}$'s satisfies the bounds
\begin{equation}\label{fc2-10}
 \sup_{z\in\mathbf{T}^{d+1}\times \mathbf{R}^{d}}|F^{(\nu)}(z)-\tilde{F}(z)|\leq C |F|_{C^{\tilde{\ell}}(\mathbf T^{d+1}\times [1,2]^d)}s_{\nu}^{\tilde{\ell}},
\end{equation}
\begin{equation}\label{fc2-11}
 \sup_{z\in \mathbf{T}^{d+1}_{2s_{\nu+1}}\times\mathbf{R}^{d}_{2s_{\nu+1}}}|F^{(\nu+1)}(z)-F^{(\nu)}(z)|\leq C |F|_{C^{\tilde{\ell}}(\mathbf T^{d+1}\times [1,2]^d)}s_{\nu}^{\tilde{\ell}},
\end{equation}
where constants $C=C(d,\tilde{\ell})$ depend on only $d$ and $\tilde{\ell}$;\\
(3) The first approximate $F^{(0)}(z)=(S_{2s_0}\tilde{F})(z)$ is ``small'' with respect to $F$. Precisely,
\begin{equation}\label{fc2-12}
 |F^{(0)}(z)|\leq C |F|_{C^{\tilde{\ell}}(\mathbf T^{d+1}\times [1,2]^d)},\ \ \forall z\in\mathbf{T}^{d+1}_{2s_0}\times\mathbf{R}^{d}_{2s_0},
\end{equation}
where constant $C=C(d,\tilde{\ell})$ is independent of $s_0$;\\
(4) From Lemma \ref{lem2-1}, we have that
  \begin{equation}\label{fc2-13}
    F(z)=F^{(0)}(z)+\sum_{\nu=0}^{\infty}(F^{(\nu+1)}(z)-F^{(\nu)}(z)),\ \  z\in \mathbf T^{d+1}\times [1,2]^d.
  \end{equation}
 Let
  \begin{equation}\label{fc2-14}
    F_0(z)=F^{(0)}(z),\ F_{\nu+1}(z)=F^{(\nu+1)}(z)-F^{(\nu)}(z).
  \end{equation}
  Then
  \begin{equation}\label{fc2-15}
  F(z)=\sum_{\nu=0}^{\infty}F_{\nu}(z),\ \ \ \ z\in \mathbf T^{d+1}\times [1,2]^d.
  \end{equation}

\section{Normal Form}\label{sec3}
Let $I_0\in [1,2]^d$ such that $\omega(I_0)=\frac{\partial H_0}{\partial I}(I_0)$ obeys Diophantine conditions (\ref{fc1-19}) and (\ref{fc1-20}).
Let $\mu_1=\frac{(a-b)^2\mu}{1000(a+b+1)(d+3)(5a-b+2ad)}$,  $\mu_2=2\mu_1=\frac{(a-b)^2\mu}{500(a+b+1)(d+3)(5a-b+2ad)}$, $m_0=10+\left[E\right]$ where $E=\max\{\frac{4B}{a-b-\frac{2(\tau_1+2)B}{\ell}-2\mu_1}, \frac{2(2\tau_1+3)(\tau_2+1)B}{B-2a-2(\tau_2+1)b-\frac{2(2\tau_1+5)(\tau_2+1)B}{\ell}-8\mu_1(\tau_2+1)-2\mu_2)}\}$ ($a$,  $b$, $\tau_1$, $\tau_2$, $B$, $\ell$ are the same as those in Theorem \ref{thm1-1}), and $[\cdot]$ is the integer part of a positive number.
 Define sequences
\begin{itemize}
\item $ \varepsilon_{j}=\varepsilon^{\frac{j B}{m_0}}, j=0,1,2,\cdots,m_0, \varepsilon_{j}=\varepsilon_{j-1}^{1+\mu_3} \ \ with \ \ \mu_3=\frac{(a-b)\mu}{10B}, j=m_0+1,m_0+2,\cdots;$
\item $ s_j=\varepsilon_{j+1}^{\frac{1}{\ell}},\ s_j^{(l)}=s_{j}-\frac{l}{10}(s_j-s_{j+1}),l=0,1,\cdots,10, j=0,1,2,\cdots;$
\item $ r_j=\varepsilon^{\frac{(j+1)(\tau_1+1) B}{\ell m_0}+\mu_1+\frac{B}{\ell}} \ with \  \mu_1=\frac{(a-b)^2\mu}{1000(a+b+1)(d+3)(5a-b+2ad)},\ \ j=0,1,2,\cdots,m_0,\ r_j=r_{j-1}^{1+\mu_3},\ j=m_0+1,m_0+2,\cdots;$
\item $ r_j^{(l)}=r_{j}-\frac{l}{10}(r_j-r_{j+1}),\ l=0,1,\cdots,10, \ j=0,1,2,\cdots;$
\item $K_j=\frac{2B}{s_j}\log\frac{1}{\varepsilon},\ j=0,1,2,\cdots;$
\item $B(r_j)=\{z\in\mathbf C^d: \, |z-I_0|\le r_j\},\  j=0,1,2,\cdots$. \end{itemize}

With the preparation of Section \ref{sec2}, we can rewrite equation (\ref{fcb1-1}) as follows:
 \begin{equation}\label{fc3-1}
H(\theta,t,I)=\frac{H_0(I)}{\varepsilon^{a}}+\frac{1}{\varepsilon^{b}} \sum_{\nu=0}^{\infty}P_{\nu}(\theta,t,I),
 \end{equation}
 where
 \begin{equation}\label{fc3-2}
P_{\nu}: \mathbf{T}^{d+1}_{2s_{\nu}}\times \mathbf{R}^{d}_{2s_{\nu}}\rightarrow\mathbf{C},
 \end{equation}
 is real analytic, and
 \begin{equation}\label{fc3-3}
\sup_{(\theta,t,I)\in  \mathbf{T}^{d+1}_{2s_{\nu}}\times \mathbf{R}^{d}_{2s_{\nu}}}|P_{\nu}|\leq C\varepsilon_{\nu}.
 \end{equation}
Let
\begin{equation}\label{fc3-4}
h^{(0)}(t,I)\equiv0,\ P^{(0)}=P_0.
\end{equation}
Then we can rewrite equation (\ref{fc3-1}) as follows:
 \begin{equation}\label{fc3-5}
H^{(0)}(\theta,t,I)=\frac{H_0(I)}{\varepsilon^{a}}+\frac{h^{(0)}(t,I)}{\varepsilon^{b}}+\frac{\varepsilon_0 P^{(0)}(\theta,t,I)}{\varepsilon^{b}}+ \sum_{\nu=1}^{\infty}\frac{P_{\nu}(\theta,t,I)}{\varepsilon^{b}}.
 \end{equation}
Define
 \begin{equation*}
 D(s,r)=\mathbf{T}^{d+1}_{s}\times B(r),\ D(s,0)=\mathbf{T}^{d+1}_{s},\ D(0,r)=B(r).
 \end{equation*}
For a function $f$ defined in $D(s,r)$ , define
 \begin{equation*}
 ||f||_{D(s,r)}=\sup_{(\theta,t,I)\in D(s,r)}|f(\theta,t,I)|.
  \end{equation*}
  Similarly, we can define $||f||_{D(0,r)}$ and $||f||_{D(s,0)}$.

Clearly, (\ref{fc3-5}) fulfill (\ref{fc3-8})-(\ref{fc3-10}) with $m=0$. Then we have the following lemma.
\begin{lem}\label{lem3-1}  Suppose that we have had $m+1$ {\rm(}$m=0,1,2,\cdots, m_0-1${\rm)} symplectic transformations $\Phi_0=id$, $\Phi_1$, $\cdots$, $\Phi_m$ with
\begin{equation}\label{fc3-6}
\Phi_j:D(s_j,r_j)\rightarrow D(s_{j-1},r_{j-1}),\ j=1,2,\cdots,m
 \end{equation}
and
\begin{equation}\label{fcf3-1}
\|\partial(\Phi_{j}-id)\|_{D(s_j,r_j)}\leq \frac{1}{2^{j+1}},\ j=1,2,\cdots,m
\end{equation}
such that system {\rm(\ref{fc3-5})} is changed by $\Phi^{(m)}=\Phi_0\circ\Phi_1\circ\cdots\circ\Phi_m$ into
\begin{equation}\label{fc3-7}
H^{(m)}=H^{(0)}\circ\Phi^{(m)}=\frac{H_0(I)}{\varepsilon^{a}}+\frac{h^{(m)}(t,I)}{\varepsilon^{b}}+\frac{\varepsilon_m P^{(m)}(\theta,t,I)}{\varepsilon^{b}}+ \sum_{\nu=m+1}^{\infty}\frac{P_{\nu}\circ\Phi^{(m)}(\theta,t,I)}{\varepsilon^{b}},
 \end{equation}
where
\begin{equation}\label{fc3-8}
\|h^{(m)}(t,I)\|_{D(s_{m}, r_{m})}\leq C,
 \end{equation}
 \begin{equation}\label{fc3-9}
 \|P^{(m)}(\theta,t,I)\|_{D(s_{m}, r_{m})}\leq C,
 \end{equation}
 \begin{equation}\label{fc3-10}
\|P_{\nu}\circ\Phi^{(m)}(\theta,t,I)\|_{D(s_{\nu}, r_{\nu})}\leq C\varepsilon_{\nu},\ \nu=m+1,m+2,\cdots.
 \end{equation}
Then there is a symplectic transformation $\Phi_{m+1}$ with
\begin{equation*}
\Phi_{m+1}:D(s_{m+1},r_{m+1})\rightarrow D(s_m,r_m)
 \end{equation*}
 and
 \begin{equation*}
\|\partial(\Phi_{m+1}-id)\|_{D(s_{m+1},r_{m+1})}\leq \frac{1}{2^{m+2}}
\end{equation*}
such that system {\rm(\ref{fc3-7})} is changed by $\Phi_{m+1}$ into {\rm($\Phi^{(m+1)}=\Phi_0\circ\Phi_1\circ\cdots\circ\Phi_{m+1}$)}
\begin{eqnarray*}\label{fc3-12}
 \nonumber H^{(m+1)}&=&H^{(m)}\circ\Phi_{m+1}=H^{(0)}\circ\Phi^{(m+1)}\nonumber\\
\nonumber&=&\frac{H_0(I)}{\varepsilon^{a}}+\frac{h^{(m+1)}(t,I)}{\varepsilon^{b}}+\frac{\varepsilon_{m+1} P^{(m+1)}(\theta,t,I)}{\varepsilon^{b}}+ \sum_{\nu=m+2}^{\infty}\frac{P_{\nu}\circ\Phi^{(m+1)}(\theta,t,I)}{\varepsilon^{b}},
 \end{eqnarray*}
where $H^{(m+1)}$ satisfies {\rm(\ref{fc3-8})}-{\rm(\ref{fc3-10})} by replacing $m$ by $m+1$.

  \end{lem}
\begin{proof}
Assume that the change $\Phi_{m+1}$ is implicitly defined by
\begin{equation}\label{fc3-16}
\Phi_{m+1}: \begin{cases}
              I=\rho+\frac{\partial S}{\partial \theta}, \\
              \phi=\theta+\frac{\partial S}{\partial \rho},\\ t=t,
    \end{cases}
\end{equation}
where $S=S(\theta, t, \rho)$ is the generating function, which will be proved to be analytic in a smaller domain
 $D(s_{m+1}, r_{m+1}).$ By a simple computation, we have
$$d I\wedge d\theta=d\rho \wedge d\theta+\sum_{i,j=1}^d\frac{\partial^{2}S}{\partial\rho_{i}\partial\theta_{j}}d\rho_{i}\wedge d\theta_{j}=d\rho \wedge d\phi.$$
Thus, the coordinates change $\Phi_{m+1}$ is symplectic if it exists. Moreover, we get the changed Hamiltonian
\begin{eqnarray}\label{fc3-17}
 \nonumber H^{(m+1)}&=&H^{(m)}\circ\Phi_{m+1}\nonumber\\
\nonumber &=&\frac{H_0(\rho+\frac{\partial S}{\partial \theta})}{\varepsilon^{a}}+\frac{h^{(m)}(t,\rho+\frac{\partial S}{\partial \theta})}{\varepsilon^{b}}+\frac{\varepsilon_{m} P^{(m)}(\theta,t,\rho+\frac{\partial S}{\partial \theta})}{\varepsilon^{b}}+\frac{\partial S}{\partial t}\\
&&+\frac{P_{m+1}\circ\Phi^{(m+1)}(\phi,t,\rho)}{\varepsilon^{b}}+ \sum_{\nu=m+2}^{\infty}\frac{P_{\nu}\circ\Phi^{(m+1)}(\phi,t,\rho)}{\varepsilon^{b}},
 \end{eqnarray}
 where $\theta=\theta(\phi, t, \rho)$ is implicitly defined by (\ref{fc3-16}). By Taylor formula, we have
\begin{eqnarray}\label{fc3-18}
 \nonumber H^{(m+1)} &=&\frac{H_0(\rho)}{\varepsilon^{a}}+\frac{h^{(m)}(t,\rho)}{\varepsilon^{b}}+\langle\frac{\omega(\rho)}{\varepsilon^a}, \frac{\partial S}{\partial \theta}\rangle+\frac{\partial S}{\partial t}+\frac{\varepsilon_{m} P^{(m)}(\theta,t,\rho)}{\varepsilon^{b}}\\
&&+\frac{P_{m+1}\circ\Phi^{(m+1)}(\phi,t,\rho)}{\varepsilon^{b}}+ \sum_{\nu=m+2}^{\infty}\frac{P_{\nu}\circ\Phi^{(m+1)}(\phi,t,\rho)}{\varepsilon^{b}}+R_1,
 \end{eqnarray}
 where $\omega(\rho)=\frac{\partial H_0}{\partial I}(\rho)$ and
 \begin{eqnarray}\label{fc3-19}
 \nonumber R_1&=&\frac{1}{\varepsilon^{a}}\int_{0}^{1}(1-\tau)\frac{\partial^2 H_0}{\partial I^2}(\rho+\tau\frac{\partial S}{\partial \theta}) (\frac{\partial S}{\partial \theta})^{2}d\tau+\frac{\varepsilon_m}{\varepsilon^b}\int_{0}^{1}\frac{\partial P^{(m)}}{\partial I}(\theta, t, \rho+\tau\frac{\partial S}{\partial\theta})\frac{\partial S}{\partial \theta}d\tau\\
&&+\frac{1}{\varepsilon^b}\int_0^1\frac{\partial h}{\partial I}(t,\rho+\tau\frac{\partial S}{\partial\theta})\frac{\partial S}{\partial \theta}\, d\tau.
\end{eqnarray}
Expanding $P^{(m)}(\theta,t,\rho)$ into a Fourier series,
\begin{equation}\label{fc3-20}
P^{(m)}(\theta,t,\rho)=\sum_{(k,l)\in\mathbf{Z}^d\times \mathbf{Z}}\widehat{{P}^{(m)}}(k,l,\rho)e^{i(\langle k,\theta\rangle+lt)}:=P_{1}^{(m)}(\theta,t,\rho)+P_{2}^{(m)}(\theta,t,\rho),
\end{equation}
where $P_{1}^{(m)}=\sum_{|k|+|l|\leq K_m}\widehat{{P}^{(m)}}(k,l,\rho)e^{i(\langle k,\theta\rangle+lt)}$, $P_{2}^{(m)}=\sum_{|k|+|l|> K_m}\widehat{{P}^{(m)}}(k,l,\rho)e^{i(\langle k,\theta\rangle+lt)}$.
Then, we derive the homological equation:
\begin{equation}\label{fc3-21}
\frac{\partial S}{\partial t}+\langle \frac{\omega(\rho)}{\varepsilon^{a}}, \frac{\partial S}{\partial\theta}\rangle
+\frac{\varepsilon_{m} P_1^{(m)}(\theta,t,\rho)}{\varepsilon^{b}}-\frac{\varepsilon_{m} \widehat{{P}_1^{(m)}}(0,t,\rho)}{\varepsilon^{b}}=0,
\end{equation}
where $\widehat{{P}_1^{(m)}}(0, t, \rho)$ is $0$-Fourier coefficient of $P_1^{(m)}(\theta,t,\rho)$ as the function of $\theta$. Let
\begin{equation}\label{fc3-22}
S(\theta,t,\rho)=\sum_{|k|+|l|\leq K_m,k\neq 0}\widehat{S}(k, l, \rho)e^{i(\langle k, \theta\rangle+lt)}.
\end{equation}
By passing to Fourier coefficients, we have
\begin{equation}\label{fc3-23}
\widehat{S}(k, l, \rho)=\frac{\varepsilon_m}{\varepsilon^b}\cdot\frac{i}{\varepsilon^{-a}\langle k, \omega(\rho)\rangle +l}\widehat{{P}^{(m)}}(k, l, \rho),\;|k|+|l|\leq K_m,\; k\in \mathbf{Z}^{d}\setminus\{0\}, l\in \mathbf{Z}.
\end{equation}
Then we can solve homological equation (\ref{fc3-21}) by setting
\begin{equation}\label{fc3-24}
S(\theta,t,\rho)=\sum_{|k|+|l|\leq K_m,k\neq 0}
\frac{\varepsilon_m}{\varepsilon^b}\cdot\frac{i}{\varepsilon^{-a}\langle k, \omega(\rho)\rangle +l}\widehat{{P}^{(m)}}(k, l, \rho)e^{i(\langle k, \theta\rangle+lt)}.
\end{equation}
By (\ref{fcb1-3}) and (\ref{fc1-19}), for $\forall \rho\in B(r_{m})$, $|k|+|l|\leq K_{m}$, $k\neq 0$, we have
\begin{eqnarray}\label{fc3-25}
|\varepsilon^{-a}\langle k, \omega(\rho)\rangle+l|
&\geq &|\varepsilon^{-a}\langle k, \omega(I_{0})\rangle+l|-|\varepsilon^{-a}\langle k, \omega(I_{0})-\omega(\rho)\rangle|\nonumber\\
&\geq & \frac{\varepsilon^{-a+\frac{B}{\ell}}\gamma}{|k|^{\tau_1}}-C\varepsilon^{-a}|k|r_{m}\nonumber\\
&\geq & \frac{\varepsilon^{-a+\frac{B}{\ell}}\gamma}{2|k|^{\tau_1}}.
\end{eqnarray}
Then, by (\ref{fc3-9}), (\ref{fc3-23})-(\ref{fc3-25}), using R\"ussmann \cite{a27, a28} subtle
arguments to give optimal estimates of small divisor series (also see Lemma 5.1 in \cite{a29}),  we get
\begin{equation}\label{fc3-26}
\|S(\theta,t,\rho)\|_{D(s^{(1)}_m, r_m)}\leq\frac{C\varepsilon^{a-b-\frac{B}{\ell}}\varepsilon_m\|P^{(m)}(\theta,t,\rho)\|_{D(s_m,r_m)}}{\gamma s_{m}^{\tau_{1}}}\leq\frac{C\varepsilon^{a-b-\frac{B}{\ell}}\varepsilon_m}{\gamma s_{m}^{\tau_{1}}}.
\end{equation}
Then by the Cauchy's estimate, we have
\begin{equation}\label{fc3-27}
 \|\frac{\partial S}{\partial \theta}\|_{D(s^{(2)}_m, r_m)}\leq \frac{C\varepsilon^{a-b-\frac{B}{\ell}}\varepsilon_m}{\gamma s_{m}^{\tau_{1}+1}}\ll r_m-r_{m+1},\ \|\frac{\partial S}{\partial \rho}\|_{D(s^{(1)}_m, r^{(1)}_m)}\leq \frac{C\varepsilon^{a-b-\frac{B}{\ell}}\varepsilon_m}{\gamma s_{m}^{\tau_{1}}r_m}\ll s_m-s_{m+1}.
\end{equation}
By (\ref{fc3-16}) and (\ref{fc3-27}) and the implicit function theorem, we get that there are analytic functions
$u=u(\phi, t, \rho), v=v(\phi, t, \rho)$ defined on
the domain $D(s^{(3)}_m, r^{(3)}_m)$ with
\begin{equation}\label{fc3-28}
\frac{\partial S(\theta,t,\rho)}{\partial \theta}=u(\phi, t, \rho),\ \frac{\partial S(\theta,t,\rho)}{\partial \rho}=-v(\phi, t, \rho)
\end{equation}
and
\begin{equation}\label{fc3-29}
 \|u\|_{D(s^{(3)}_m, r^{(3)}_m)}\leq \frac{C\varepsilon^{a-b-\frac{B}{\ell}}\varepsilon_m}{\gamma s_{m}^{\tau_{1}+1}}\ll r_m-r_{m+1},\ \|v\|_{D(s^{(3)}_m, r^{(3)}_m)}\leq \frac{C\varepsilon^{a-b-\frac{B}{\ell}}\varepsilon_m}{\gamma s_{m}^{\tau_{1}}r_m}\ll s_m-s_{m+1}
\end{equation}
such that
\begin{equation}\label{fc3-30}
\Phi_{m+1}: \begin{cases}
              I=\rho+u(\phi, t, \rho), \\
              \theta=\phi+v(\phi, t, \rho),\\ t=t.
    \end{cases}
\end{equation}
Then, we have
\begin{equation}\label{fc3-31}
\Phi_{m+1}(D(s_{m+1},r_{m+1}))\subseteq \Phi_{m+1}(D(s^{(3)}_{m},r^{(3)}_{m}))\subseteq D(s_{m},r_{m}).
\end{equation}

Let
\begin{equation}\label{fc3-32}
h^{(m+1)}(t,\rho)=h^{(m)}(t,\rho)+\varepsilon_{m}\widehat{{P}_1^{(m)}}(0, t, \rho),
\end{equation}
\begin{equation}\label{fc3-33}
 \frac{\varepsilon_{m+1}P^{(m+1)}(\phi,t,\rho)}{\varepsilon^{b}}=\frac{\varepsilon_{m} P_2^{(m)}(\theta,t,\rho)}{\varepsilon^{b}}+\frac{P_{m+1}\circ\Phi^{(m+1)}(\phi,t,\rho)}{\varepsilon^{b}}+R_1.
\end{equation}
Then by (\ref{fc3-18}), (\ref{fc3-20}), (\ref{fc3-21}), (\ref{fc3-32}) and (\ref{fc3-33}), we have
\begin{equation}\label{fc3-34}
H^{(m+1)}(\phi,t,\rho)=\frac{H_0(\rho)}{\varepsilon^{a}}+\frac{h^{(m+1)}(t,\rho)}{\varepsilon^{b}}+\frac{\varepsilon_{m+1} P^{(m+1)}(\phi,t,\rho)}{\varepsilon^{b}}+ \sum_{\nu=m+2}^{\infty}\frac{P_{\nu}\circ\Phi^{(m+1)}(\phi,t,\rho)}{\varepsilon^{b}}.
 \end{equation}
 By (\ref{fc3-9}) and (\ref{fc3-20}), it is not difficult to show that (see Lemma A.2 in \cite{a30}),
\begin{equation}\label{fc3-36}
\| \frac{\varepsilon_{m} P_2^{(m)}(\theta,t,\rho)}{\varepsilon^{b}}\|_{D(s^{(9)}_{m}, r^{(9)}_{m})}\leq \frac{C\varepsilon_{m}}{\varepsilon^{b}}K_{m}^{d+1}e^{-\frac{K_{m}s_{m}}{2}}\leq\frac{C\varepsilon_{m+1}}{\varepsilon^{b}}.
\end{equation}
By (\ref{fc3-8}), (\ref{fc3-9}), (\ref{fc3-20}), (\ref{fc3-32}) and (\ref{fc3-36}), we have
\begin{equation}\label{fc3-35}
\| h^{(m+1)}\|_{D(s_{m+1}, r_{m+1})}\leq \| h^{(m)}\|_{D(s_{m+1}, r_{m+1})}+\| \varepsilon_{m}\widehat{{P}_1^{(m)}}(0, t, \rho)\|_{D(s_{m+1}, r_{m+1})}\leq C.
\end{equation}

By (\ref{fc3-8}), (\ref{fc3-9}), (\ref{fc3-26})-(\ref{fc3-29}), we have
\begin{equation}\label{fc3-37}
\| \frac{1}{\varepsilon^{a}}\int_{0}^{1}(1-\tau)\frac{\partial^2 H_0}{\partial I^2}(\rho+\tau\frac{\partial S}{\partial \theta}) (\frac{\partial S}{\partial \theta})^{2}d\tau\|_{D(s^{(9)}_{m}, r^{(9)}_{m})}\leq \frac{C}{\varepsilon^{a}r_m^2}\cdot(\frac{\varepsilon^{a-b-\frac{B}{\ell}}\varepsilon_m}{\gamma s_{m}^{\tau_{1}+1}})^2\leq\frac{C\varepsilon_{m+1}}{\varepsilon^{b}},
\end{equation}
\begin{equation}\label{fc3-38}
\| \frac{\varepsilon_m}{\varepsilon^b}\int_{0}^{1}\frac{\partial P^{(m)}}{\partial I}(\theta, t, \rho+\tau\frac{\partial S}{\partial\theta})\frac{\partial S}{\partial \theta}d\tau\|_{D(s^{(9)}_{m}, r^{(9)}_{m})}\leq \frac{C\varepsilon_m}{\varepsilon^{b}r_m}\cdot\frac{\varepsilon^{a-b-\frac{B}{\ell}}\varepsilon_m}{\gamma s_{m}^{\tau_{1}+1}}\leq\frac{C\varepsilon_{m+1}}{\varepsilon^{b}},
\end{equation}
\begin{equation}\label{fc3-39}
\| \frac{1}{\varepsilon^b}\int_0^1\frac{\partial h}{\partial I}(t,\rho+\tau\frac{\partial S}{\partial\theta})\frac{\partial S}{\partial \theta}\, d\tau\|_{D(s^{(9)}_{m}, r^{(9)}_{m})}\leq \frac{C}{\varepsilon^{b}r_m}\cdot\frac{\varepsilon^{a-b-\frac{B}{\ell}}\varepsilon_m}{\gamma s_{m}^{\tau_{1}+1}}\leq\frac{C\varepsilon_{m+1}}{\varepsilon^{b}}.
\end{equation}

By (\ref{fc3-29}) and (\ref{fc3-30}), we have
\begin{equation}\label{fc3-40}
\Phi_{m+1}(\phi,t,\rho)=(\theta,t,I),\ (\phi,t,\rho)\in D(s_m^{(3)},r_m^{(3)}).
\end{equation}
By (\ref{fc3-29}), (\ref{fc3-30}) and (\ref{fc3-40}), we have
\begin{equation}\label{fc3-41}
\|I-\rho\|_{D(s^{(3)}_m,r^{(3)}_m)}\leq \frac{C\varepsilon^{a-b-\frac{B}{\ell}}\varepsilon_m}{\gamma s_{m}^{\tau_{1}+1}}, \ \ \|\theta-\phi\|_{D(s^{(3)}_m,r^{(3)}_m)}\leq \frac{C\varepsilon^{a-b-\frac{B}{\ell}}\varepsilon_m}{\gamma s_{m}^{\tau_{1}}r_m}.
\end{equation}
By (\ref{fc3-30}), (\ref{fc3-41}) and Cauchy's estimate, we have
\begin{equation}\label{fc3-42}
\|\partial(\Phi_{m+1}-id)\|_{D(s^{(4)}_m,r_m^{(4)})}\leq \frac{C\varepsilon^{a-b-\frac{B}{\ell}}\varepsilon_m}{\gamma s_{m}^{\tau_{1}+1}r_m}.
\end{equation}
It follows that
\begin{equation}\label{fc3-43}
\|\partial(\Phi_{m+1}-id)\|_{D(s_{m+1},r_{m+1})}\leq \frac{1}{2^{m+2}}.
\end{equation}
By (\ref{fc3-6}), (\ref{fcf3-1}), (\ref{fc3-31}) and (\ref{fc3-43}), we have
 \begin{eqnarray}\label{fc3-44}
 \nonumber &&\|\partial\Phi^{(m+1)}(\phi,t,\rho)\|_{D(s_{m+1},r_{m+1})}\\
 \nonumber&=&\|(\partial\Phi_1\circ\Phi_2\circ\cdots\circ\Phi_{m+1})(\partial\Phi_2\circ\Phi_3\circ\cdots\circ\Phi_{m+1})\cdots(\partial\Phi_{m+1})\|_{D(s_{m+1},r_{m+1})}\\
   \nonumber&\leq&\prod_{j=0}^{m}(1+\frac{1}{2^{j+2}})\\
   &\leq&2.
\end{eqnarray}
It follows that
\begin{equation}\label{fc3-45}
\Phi^{(m+1)}(D(s_{\nu},r_{\nu}))\subset \mathbf{T}^{d+1}_{2s_{\nu}}\times \mathbf{R}^{d}_{2s_{\nu}},\ \nu=m+1,m+2,\cdots.
\end{equation}
In fact,  suppose that $w=\Phi^{(m+1)}(z)$ with $z=(\phi,t,\rho)\in D(s_{\nu},r_{\nu})$.
Since $\Phi^{(m+1)}$ is real for real argument and $r_{\nu}<s_{\nu}$,  we have
 \begin{eqnarray}\label{fc3-46}
 \nonumber &&|{\rm Im} w|=|{\rm Im} \Phi^{(m+1)}(z)|=|{\rm Im} \Phi^{(m+1)}(z)-{\rm Im} \Phi^{(m+1)}({\rm Re}z)|\\
 \nonumber&\leq&| \Phi^{(m+1)}(z)- \Phi^{(m+1)}({\rm Re}z)|\\
   \nonumber&\leq&\|\partial\Phi^{(m+1)}(\phi,t,\rho)\|_{D(s_{m+1},r_{m+1})}|{\rm Im} z|\\
   &\leq&2|{\rm Im} z|\leq2s_{\nu}.
\end{eqnarray}
By (\ref{fc3-3}) and (\ref{fc3-45}), we have
\begin{equation}\label{fc3-47}
\|\frac{P_{m+1}\circ\Phi^{(m+1)}(\phi,t,\rho)}{\varepsilon^{b}}\|_{D(s_{m+1},r_{m+1})}\leq\frac{C\varepsilon_{m+1}}{\varepsilon^b},
\end{equation}
\begin{equation}\label{fc3-48}
\|P_{\nu}\circ\Phi^{(m+1)}(\phi,t,\rho)\|_{D(s_{\nu},r_{\nu})}\leq C\varepsilon_{\nu},\ \nu=m+2,m+3,\cdots.
\end{equation}
By (\ref{fc3-19}), (\ref{fc3-29}), (\ref{fc3-33}), (\ref{fc3-36}), (\ref{fc3-37})-(\ref{fc3-39}) and (\ref{fc3-47}), we have
 \begin{equation}\label{fc3-49}
 \|P^{(m+1)}(\phi,t,\rho)\|_{D(s_{m+1}, r_{m+1})}\leq C.
 \end{equation}

 The proof  is finished by  (\ref{fc3-31}), (\ref{fc3-34}), (\ref{fc3-35}), (\ref{fc3-43}), (\ref{fc3-48}) and (\ref{fc3-49}).
\end{proof}

 By Lemma \ref{lem3-1}, there is a symplectic transformation $\Phi^{(m_0)}=\Phi_{0}\circ\Phi_{1}\circ\cdots\circ\Phi_{m_0}$ with
 $$\Phi^{(m_0)}: D(s_{m_0}, r_{m_0})\rightarrow D(s_{0}, r_{0})$$
 such that system (\ref{fc3-5}) is changed by $\Phi^{(m_0)}$ into
 \begin{equation}\label{fc3-50}
H^{(m_0)}=\frac{H_0(I)}{\varepsilon^{a}}+\frac{h^{(m_0)}(t,I)}{\varepsilon^{b}}+\frac{\varepsilon^B P^{(m_0)}(\theta,t,I)}{\varepsilon^{b}}+ \sum_{\nu=m_0+1}^{\infty}\frac{P_{\nu}\circ\Phi^{(m_0)}(\theta,t,I)}{\varepsilon^{b}}
 \end{equation}
 where
\begin{equation}\label{fc3-51}
\|h^{(m_0)}(t,I)\|_{D(s_{m_0}, r_{m_0})}\leq C,
 \end{equation}
 \begin{equation}\label{fc3-52}
 \|P^{(m_0)}(\theta,t,I)\|_{D(s_{m_0}, r_{m_0})}\leq C,
 \end{equation}
 \begin{equation}\label{fc3-53}
\|P_{\nu}\circ\Phi^{(m_0)}(\theta,t,I)\|_{D(s_{\nu}, r_{\nu})}\leq C\varepsilon_{\nu},\ \nu=m_0+1,m_0+2,\cdots.
 \end{equation}
\section{A symplectic transformation}\label{sec4}
Let $[h^{(m_0)}](I)=\widehat{h^{(m_0)}} (0,I)$ be the  $0$-Fourier coefficient of $h^{(m_0)}(t,I)$ as the function of $t$.
In order to eliminate the dependence of $h^{(m_0)} (t,I)$ on the time-variable $t$, we introduce the following transformation
 \begin{equation}\label{fcc4-1}
\Psi:\ \rho=I,\ \phi=\theta+\frac{\partial \tilde S(t,I)}{\partial I},
\end{equation}
where $\tilde S(t,I)=\frac{1}{\varepsilon^b}\int_0^t\left( [ h^{(m_0)}](I)-h^{(m_0)}(\xi,I) \right) d \xi.$ It is symplectic by  easy verification $d\,\rho\wedge d\phi=d\, I\wedge d\, \theta$.
Noting that the transformation is not small. So $\Psi$ is not close to the Identity. Let  \begin{equation*}\label{fc4-1}
\tilde{s}_0=\varepsilon^{b+\frac{(m_0+1)(2\tau_1+3) B}{\ell m_0}+4\mu_1+\frac{2B}{\ell}}, \ \tilde{r}_{0}=\varepsilon^{a+(\tau_2+1)b+\frac{(m_0+1)(2\tau_1+3)(\tau_2+1)B}{m_0\ell}+4\mu_1(\tau_2+1)+\mu_2+\frac{2B(\tau_2+1)}{\ell}},
\end{equation*}
where $\mu_1=\frac{(a-b)^2\mu}{1000(a+b+1)(d+3)(5a-b+2ad)}$,  $\mu_2=2\mu_1$.
We introduce a domain
$$\mathcal{D}:=\left\{t=t_1+t_2i\in \mathbf T_{s_{m_0}}:\; |t_2|\le \tilde{s}_0 \right\}\times\left\{I=I_1+I_2i\in B(r_{m_0}):\;  |I_2|\le\tilde{r}_{0}\right\},$$
where $t_1,t_2,I_1,I_2$ are real numbers. Noting that $h^{(m_0)}(t,I)$ is real for real arguments. Thus,  for $(t,I)\in \mathcal{D}$,  we have
\begin{eqnarray}\label{fc4-2}
 \nonumber&&\|{\rm Im}\frac{\partial \tilde S(t,I)}{\partial I}\|_{\mathcal{D}}\\
 \nonumber &=&\|{\rm Im}\frac{\partial \tilde S}{\partial I}(t_1+t_2i,I_1+I_2i)-{\rm Im}\frac{\partial \tilde S}{\partial I}(t_1,I_1)\|_{\mathcal{D}}\\
\nonumber&\leq&\|\frac{\partial \tilde S}{\partial I}(t_1+t_2i,I_1+I_2i)-\frac{\partial \tilde S}{\partial I}(t_1,I_1)\|_{\mathcal{D}}\\
\nonumber&\leq&\|\frac{\partial^2 \tilde S(t,I)}{\partial I \partial t}\|_{\mathcal{D}}\|t_2i\|_{\mathcal{D}}+\|\frac{\partial^2 \tilde S(t,I)}{\partial^2 I}\|_{\mathcal{D}}\|I_2i\|_{\mathcal{D}}\\
\nonumber&\leq&\frac{C\tilde{s}_0}{\varepsilon^br_{m_0}s_{m_0}}+\frac{C\tilde{r}_0}{\varepsilon^br^2_{m_0}}\\
&\leq&\frac{1}{2}s_{m_0}.
\end{eqnarray}
By (\ref{fc3-50}), (\ref{fcc4-1}) and (\ref{fc4-2}), we have
 \begin{equation}\label{fc4-3}
\Psi(\mathbf T^{d}_{s_{m_0}/2}\times\mathcal D)\subset D(s_{m_0},r_{m_0})
\end{equation}
and
\begin{eqnarray}\label{fc4-4}
 \nonumber \tilde{H}(\phi,t,\rho)&=&H^{(m_0)}\circ\Psi\nonumber\\
&=&\frac{H_0(\rho)}{\varepsilon^{a}}+\frac{[ h^{(m_0)}](\rho)}{\varepsilon^{b}}+\frac{\varepsilon^B \breve{P}^{(m_0)}(\phi,t,\rho)}{\varepsilon^{b}}+ \sum_{\nu=m_0+1}^{\infty}\frac{P_{\nu}\circ\Phi^{(m_0)}\circ\Psi(\phi,t,\rho)}{\varepsilon^{b}},
 \end{eqnarray}
 where $\breve{P}^{(m_0)}(\phi,t,\rho)=P^{(m_0)}(\phi-\frac{\partial}{\partial I}\tilde S(t,\rho),t,\rho)$ and $\|\breve{P}^{(m_0)}\|_{\mathbf T^{d}_{s_{m_0}/2}\times\mathcal D}\leq C$.

\section{Iterative lemma}\label{sec5}
 By (\ref{fc3-51}), we have
 \begin{equation}\label{fc5-1}
\varepsilon^{a-b}\|\frac{\partial^2 [ h^{(m_0)}](\rho)}{\partial \rho^2} \|_{D(0,\frac{r_{m_0}}{2})}\leq\frac{C\varepsilon^{a-b}}{r^2_{m_0}}\ll1.
\end{equation}
Then by (\ref{fcb1-3}), (\ref{fc3-51}) and (\ref{fc5-1}), solving the equation $\frac{\partial H_0(\rho)}{\partial \rho}+\varepsilon^{a-b}\frac{\partial [ h^{(m_0)}](\rho)}{\partial \rho} =\omega(I_0)$ by Newton iteration, we get that there exists $\tilde{I}_{0}\in \mathbf{R}^{d}\cap D(0,\frac{r_{m_0}}{2})$ with $|\tilde{I}_{0}-I_0|\leq\frac{C\varepsilon^{a-b}}{r_{m_0}}\ll r_{m_0}$ such that
 \begin{equation}\label{fc5-2}
\frac{\partial H_0}{\partial \rho}(\tilde{I}_{0})+\varepsilon^{a-b}\frac{\partial [ h^{(m_0)}]}{\partial \rho} (\tilde{I}_{0})=\omega(I_0),
\end{equation}
where $\omega(I_0)=\frac{\partial H_0}{\partial \rho}(I_{0})$. For any $c>0$ and any  $y_0\in\mathbf{{R}}^{d}$, let
 \begin{equation*}
 B(y_0,c)=\{z\in\mathbf C^d: \, |z-y_0|\le c\}.
 \end{equation*}
 Define
 \begin{equation*}
 \tilde{D}(s,r(I))=\mathbf{T}^{d+1}_{s}\times B(I,r),\ \tilde{D}(s,0)=\mathbf{T}^{d+1}_{s},\ \tilde{D}(0,r(I))=B(I,r).
 \end{equation*}
Let $\tilde{\varepsilon}_0=\varepsilon_{m_0}=\varepsilon^B$. Noting that $|\tilde{I}_{0}-I_0|\ll r_{m_0}$, and by (\ref{fc4-3}), we have
 \begin{equation}\label{fbc5-1}
\Psi(\tilde{D}(\tilde{s}_0,\tilde{r}_0(\tilde{I}_0)))\subset D(s_{m_0},r_{m_0}).
\end{equation}
Then we can rewrite equation (\ref{fc4-4}) as follows:
 \begin{equation}\label{fc5-3}
\tilde{H}^{(0)}(\theta,t,I)=\frac{H^{(0)}_0(I)}{\varepsilon^{a}}+ \frac{\tilde{P}^{(0)}(\theta,t,I)}{\varepsilon^{b}}+ \sum_{\nu=m_0+1}^{\infty}\frac{P_{\nu}\circ\Phi^{(m_0)}\circ\Psi(\theta,t,I)}{\varepsilon^{b}},
 \end{equation}
where $(\theta,t,I)\in \tilde{D}(\tilde{s}_0,\tilde{r}_0(\tilde{I}_0))$, $H^{(0)}_0(I)=H_0(I)+\varepsilon^{a-b}[ h^{(m_0)}](I)$, $\tilde{P}^{(0)}=\varepsilon^B \breve{P}^{(m_0)}$ and
 \begin{equation}\label{fc5-4}
\frac{\partial H^{(0)}_0}{\partial I}(\tilde{I}_{0})=\omega(I_0),
\end{equation}
 \begin{equation}\label{fc5-5}
\|\tilde{P}^{(0)}\|_{\tilde{D}(\tilde{s}_0,\tilde{r}_0(\tilde{I}_0))}\leq C\tilde{\varepsilon}_0.
\end{equation}
 By (\ref{fcb1-3}) and (\ref{fc5-1}), we get that there exist constants $M_0>0$, $h_0>0$ such that
\begin{equation}\label{fc5-6}
det\left(\frac{\partial^2 H^{(0)}_0(I)}{\partial I^2}\right),\ det\left(\frac{\partial^2 H^{(0)}_0(I)}{\partial I^2}\right)^{-1}\leq M_0, \ \forall I\in \tilde{D}(0,\tilde{r}_0(\tilde{I}_0))
 \end{equation}
 and
 \begin{equation}\label{fc5-7}
\|\frac{\partial^2 H^{(0)}_0(I)}{\partial I^2}\|_{\tilde{D}(0,\tilde{r}_0(\tilde{I}_0))}\leq h_0.
 \end{equation}
  Define sequences
\begin{itemize}
\item $ \tilde{\varepsilon}_0=\varepsilon_{m_0}=\varepsilon^B,\ \tilde{\varepsilon}_{j+1}=\tilde{\varepsilon}_{j}^{1+\mu_3}=\varepsilon_{m_0+1+j} \ with \ \mu_3=\frac{(a-b)\mu}{10B},\  j=0,1,\cdots;$
\item $ \tilde{s}_0=\varepsilon^{b+\frac{(m_0+1)(2\tau_1+3) B}{\ell m_0}+4\mu_1+\frac{2B}{\ell}}  \ \ \ with  \ \ \ \mu_1=\frac{(a-b)^2\mu}{1000(a+b+1)(d+3)(5a-b+2ad)},\ \ \tilde{s}_{j+1}= \tilde{s}_{j}^{1+\mu_3}, \ \   \tilde{s}_j^{(l)}=\tilde{s}_{j}-\frac{l}{10}(\tilde{s}_j-\tilde{s}_{j+1}),\ l=0,1,\cdots,10, \ j=0,1,2,\cdots;$
\item $\tilde{r}_{0}=\varepsilon^{a+(\tau_2+1)b+\frac{(m_0+1)(2\tau_1+3)(\tau_2+1)B}{m_0\ell}+4\mu_1(\tau_2+1)+\mu_2+\frac{2B(\tau_2+1)}{\ell}} \ \ with \ \  \mu_2=2\mu_1, \ \  \tilde{r}_{j+1}= \tilde{r}_{j}^{1+\mu_3},\  \tilde{r}_j^{(l)}=\tilde{r}_{j}-\frac{l}{10}(\tilde{r}_j-\tilde{r}_{j+1}),\ l=0,1,\cdots,10, \ j=0,1,2,\cdots;$
\item $\tilde{K}_j=\frac{2}{\tilde{s}_j}\log\frac{1}{\tilde{\varepsilon}_j},\ j=0,1,2,\cdots;$
\item $h_j=h_0(2-\frac{1}{2^j}),\ j=0,1,2,\cdots;$
\item $M_j=M_0(2-\frac{1}{2^j}),\ j=0,1,2,\cdots.$
 \end{itemize}

 We claim that  \begin{equation}\label{fc5-8}
\|P_{\nu}\circ\Phi^{(m_0)}\circ\Psi(\theta,t,I)\|_{\tilde{D}(\tilde{s}_{\nu-m_0},\tilde{r}_{\nu-m_0}(\tilde{I}_0))}\leq C\varepsilon_{\nu}= C\tilde{\varepsilon}_{\nu-m_0},\ \nu=m_0+1,m_0+2,\cdots.
 \end{equation}
 In fact, for $(t,I)=(t_1+t_2i,I_1+I_2i)\in \tilde{D}(\tilde{s}_{\nu-m_0},\tilde{r}_{\nu-m_0}(\tilde{I}_0))$, where $t_1,t_2,I_1,I_2$ are real numbers,  we have
\begin{eqnarray}\label{fc5-9}
 \nonumber&&\|{\rm Im}\frac{\partial \tilde S(t,I)}{\partial I}\|_{\tilde{D}(\tilde{s}_{\nu-m_0},\tilde{r}_{\nu-m_0}(\tilde{I}_0))}\\
 \nonumber &=&\|{\rm Im}\frac{\partial \tilde S}{\partial I}(t_1+t_2i,I_1+I_2i)-{\rm Im}\frac{\partial \tilde S}{\partial I}(t_1,I_1)\|_{\tilde{D}(\tilde{s}_{\nu-m_0},\tilde{r}_{\nu-m_0}(\tilde{I}_0))}\\
\nonumber&\leq&\|\frac{\partial \tilde S}{\partial I}(t_1+t_2i,I_1+I_2i)-\frac{\partial \tilde S}{\partial I}(t_1,I_1)\|_{\tilde{D}(\tilde{s}_{\nu-m_0},\tilde{r}_{\nu-m_0}(\tilde{I}_0))}\\
\nonumber&\leq&\|\frac{\partial^2 \tilde S(t,I)}{\partial I \partial t}\|_{\mathcal{D}}\|t_2i\|_{\tilde{D}(\tilde{s}_{\nu-m_0},\tilde{r}_{\nu-m_0}(\tilde{I}_0))}+\|\frac{\partial^2 \tilde S(t,I)}{\partial^2 I}\|_{\mathcal{D}}\|I_2i\|_{\tilde{D}(\tilde{s}_{\nu-m_0},\tilde{r}_{\nu-m_0}(\tilde{I}_0))}\\
\nonumber&\leq&\frac{C\tilde{s}_{\nu-m_0}}{\varepsilon^br_{m_0}s_{m_0}}+\frac{C\tilde{r}_{\nu-m_0}}{\varepsilon^br^2_{m_0}}\\
&\leq&\frac{1}{2}s_{\nu}.
\end{eqnarray}
It follows that
 \begin{equation}\label{fc5-10}
\Psi(\tilde{D}(\tilde{s}_{\nu-m_0},\tilde{r}_{\nu-m_0}(\tilde{I}_0)))\subset \tilde{D}(s_{\nu},\tilde{r}_{\nu-m_0}(\tilde{I}_0)).
\end{equation}
Suppose that $w=\Phi^{(m_0)}(z)$ with $z=(\theta,t,I)\in \tilde{D}(s_{\nu},\tilde{r}_{\nu-m_0}(\tilde{I}_0))\subset D(s_{m_0},r_{m_0})$.
Since $\Phi^{(m_0)}$ is real for real argument and $\tilde{r}_{\nu-m_0}<r_{\nu}<s_{\nu}$, then by (\ref{fc3-44}) with $m=m_0-1$, we have
 \begin{eqnarray}\label{fc5-11}
 \nonumber &&|{\rm Im} w|=|{\rm Im} \Phi^{(m_0)}(z)|=|{\rm Im} \Phi^{(m_0)}(z)-{\rm Im} \Phi^{(m_0)}({\rm Re}z)|\\
 \nonumber&\leq&| \Phi^{(m_0)}(z)- \Phi^{(m_0)}({\rm Re}z)|\\
   \nonumber&\leq&\|\partial\Phi^{(m_0)}(\theta,t,I)\|_{D(s_{m_0},r_{m_0})}|{\rm Im} z|\\
   &\leq&2|{\rm Im} z|\leq2s_{\nu}.
\end{eqnarray}
Then by (\ref{fc5-10}) and (\ref{fc5-11}), we have
\begin{equation}\label{fc5-12}
\Phi^{(m_0)}\circ\Psi(\tilde{D}(\tilde{s}_{\nu-m_0},\tilde{r}_{\nu-m_0}(\tilde{I}_0)))\subset \mathbf{T}^{d+1}_{2s_{\nu}}\times \mathbf{R}^{d}_{2s_{\nu}},\ \nu=m_0+1,m_0+2,\cdots.
\end{equation}
By (\ref{fc3-3}) and (\ref{fc5-12}), the proof of (\ref{fc5-8}) is completed.  Clearly, by (\ref{fc5-4})-(\ref{fc5-8}), (\ref{fc5-3}) fulfill (\ref{fc5-15})-(\ref{fc5-19}) with $m=0$. Then we have the following lemma.

\begin{lem}\label{lem5-1}{\rm (Iterative Lemma)}  Suppose that we have had $m+1$ {\rm(}$m=0,1,2,\cdots${\rm)} symplectic transformations $\tilde{\Phi}_0=id$, $\tilde{\Phi}_1$, $\cdots$, $\tilde{\Phi}_m$ with
\begin{equation}\label{fc5-13}
\tilde{\Phi}_j:\tilde{D}(\tilde{s}_j,\tilde{r}_j(\tilde{I}_j))\rightarrow \tilde{D}(\tilde{s}_{j-1},\tilde{r}_{j-1}(\tilde{I}_{j-1})),\ j=1,2,\cdots,m
 \end{equation}
 and
\begin{equation}\label{fcg5-1}
\|\partial(\tilde{\Phi}_{j}-id)\|_{\tilde{D}(\tilde{s}_j,\tilde{r}_j(\tilde{I}_j))}\leq \frac{1}{2^{j+1}},\ j=1,2,\cdots,m
\end{equation}
 where $\tilde{I}_j\in \mathbf{R^{d}}, \ j=0,1,2,\cdots,m$ such that system {\rm(\ref{fc5-3})} is changed by $\tilde{\Phi}^{(m)}=\tilde{\Phi}_0\circ\tilde{\Phi}_1\circ\cdots\circ\tilde{\Phi}_m$ into
\begin{equation}\label{fc5-14}
\tilde{H}^{(m)}=\tilde{H}^{(0)}\circ\tilde{\Phi}^{(m)}=\frac{H^{(m)}_0(I)}{\varepsilon^{a}}+ \frac{\tilde{P}^{(m)}(\theta,t,I)}{\varepsilon^{b}}+ \sum_{\nu=m_0+m+1}^{\infty}\frac{P_{\nu}\circ\Phi^{(m_0)}\circ\Psi\circ\tilde{\Phi}^{(m)}(\theta,t,I)}{\varepsilon^{b}},
 \end{equation}
where
 \begin{equation}\label{fc5-15}
\frac{\partial H^{(m)}_0}{\partial I}(\tilde{I}_{m})=\omega(I_0),
\end{equation}
 \begin{equation}\label{fc5-16}
\|\tilde{P}^{(m)}\|_{\tilde{D}(\tilde{s}_m,\tilde{r}_m(\tilde{I}_m))}\leq C\tilde{\varepsilon}_m,
\end{equation}
\begin{equation}\label{fc5-17}
det\left(\frac{\partial^2 H^{(m)}_0(I)}{\partial I^2}\right),\ det\left(\frac{\partial^2 H^{(m)}_0(I)}{\partial I^2}\right)^{-1}\leq M_m, \ \forall I\in \tilde{D}(0,\tilde{r}_m(\tilde{I}_m)),
 \end{equation}
 \begin{equation}\label{fc5-18}
\|\frac{\partial^2 H^{(m)}_0(I)}{\partial I^2}\|_{\tilde{D}(0,\tilde{r}_m(\tilde{I}_m))}\leq h_m,
 \end{equation}
 \begin{equation}\label{fc5-19}
\|P_{\nu}\circ\Phi^{(m_0)}\circ\Psi\circ\tilde{\Phi}^{(m)}(\theta,t,I)\|_{\tilde{D}(\tilde{s}_{\nu-m_0},\tilde{r}_{\nu-m_0}(\tilde{I}_m))}\leq  C\tilde{\varepsilon}_{\nu-m_0},\ \nu=m_0+m+1,m_0+m+2,\cdots.
 \end{equation}
Then there is a symplectic transformation $\tilde{\Phi}_{m+1}$ with
\begin{equation}\label{fc5-20}
\tilde{\Phi}_{m+1}:\tilde{D}(\tilde{s}_{m+1},\tilde{r}_{m+1}(\tilde{I}_{m+1}))\rightarrow \tilde{D}(\tilde{s}_m,\tilde{r}_m(\tilde{I}_m))
 \end{equation}
  and
\begin{equation*}
\|\partial(\tilde{\Phi}_{m+1}-id)\|_{\tilde{D}(\tilde{s}_{m+1},\tilde{r}_{m+1}(\tilde{I}_{m+1}))}\leq \frac{1}{2^{m+2}}
\end{equation*}
where $\tilde{I}_{m+1}\in \mathbf{R^{d}}$ such that system {\rm(\ref{fc5-14})} is changed by $\tilde{\Phi}_{m+1}$ into {\rm($\tilde{\Phi}^{(m+1)}=\tilde{\Phi}_0\circ\tilde{\Phi}_1\circ\cdots\circ\tilde{\Phi}_{m+1}$)}
\begin{eqnarray*}
 \nonumber \tilde{H}^{(m+1)}&=&\tilde{H}^{(m)}\circ\tilde{\Phi}_{m+1}=\tilde{H}^{(0)}\circ\tilde{\Phi}^{(m+1)}\nonumber\\
\nonumber&=&\frac{H^{(m+1)}_0(I)}{\varepsilon^{a}}+ \frac{\tilde{P}^{(m+1)}(\theta,t,I)}{\varepsilon^{b}}+ \sum_{\nu=m_0+m+2}^{\infty}\frac{P_{\nu}\circ\Phi^{(m_0)}\circ\Psi\circ\tilde{\Phi}^{(m+1)}(\theta,t,I)}{\varepsilon^{b}},
 \end{eqnarray*}
 where $\tilde{H}^{(m+1)}$ satisfies {\rm(\ref{fc5-15})}-{\rm(\ref{fc5-19})} by replacing $m$ by $m+1$.

  \end{lem}
  \begin{proof}
Assume that the change $\tilde{\Phi}_{m+1}$ is implicitly defined by
\begin{equation}\label{fc5-22}
\tilde{\Phi}_{m+1}: \begin{cases}
              I=\rho+\frac{\partial S}{\partial \theta}, \\
              \phi=\theta+\frac{\partial S}{\partial \rho},\\ t=t,
    \end{cases}
\end{equation}
where $S=S(\theta, t, \rho)$ is the generating function, which will be proved to be analytic in a smaller domain
 $\tilde{D}(\tilde{s}_{m+1}, \tilde{r}_{m+1}(\tilde{I}_{m+1})).$ By a simple computation, we have
$$d I\wedge d\theta=d\rho \wedge d\theta+\sum_{i,j=1}^d\frac{\partial^{2}S}{\partial\rho_{i}\partial\theta_{j}}d\rho_{i}\wedge d\theta_{j}=d\rho \wedge d\phi.$$
Thus, the coordinates change $\tilde{\Phi}_{m+1}$ is symplectic if it exists. Moreover, we get the changed Hamiltonian
\begin{eqnarray}\label{fc5-23}
 \nonumber \tilde{H}^{(m+1)}&=&\tilde{H}^{(m)}\circ\tilde{\Phi}_{m+1}\nonumber\\
\nonumber &=&\frac{H^{(m)}_0(\rho+\frac{\partial S}{\partial \theta})}{\varepsilon^{a}}+\frac{ \tilde{P}^{(m)}(\theta,t,\rho+\frac{\partial S}{\partial \theta})}{\varepsilon^{b}}+\frac{P_{m_0+m+1}\circ\Phi^{(m_0)}\circ\Psi\circ\tilde{\Phi}^{(m+1)}(\phi,t,\rho)}{\varepsilon^{b}}\\
&&+\frac{\partial S}{\partial t}+\sum_{\nu=m_0+m+2}^{\infty}\frac{P_{\nu}\circ\Phi^{(m_0)}\circ\Psi\circ\tilde{\Phi}^{(m+1)}(\phi,t,\rho)}{\varepsilon^{b}},
 \end{eqnarray}
 where $\theta=\theta(\phi, t, \rho)$ is implicitly defined by (\ref{fc5-22}). By Taylor formula, we have
 \begin{eqnarray}\label{fc5-24}
 \nonumber \tilde{H}^{(m+1)} &=&\frac{H^{(m)}_0(\rho)}{\varepsilon^{a}}+\langle\frac{\omega^{(m)}(\rho)}{\varepsilon^a}, \frac{\partial S}{\partial \theta}\rangle+\frac{\partial S}{\partial t}+\frac{ \tilde{P}^{(m)}(\theta,t,\rho)}{\varepsilon^{b}}+R_1\\
 \nonumber &&+\frac{P_{m_0+m+1}\circ\Phi^{(m_0)}\circ\Psi\circ\tilde{\Phi}^{(m+1)}(\phi,t,\rho)}{\varepsilon^{b}}\\
&&+\sum_{\nu=m_0+m+2}^{\infty}\frac{P_{\nu}\circ\Phi^{(m_0)}\circ\Psi\circ\tilde{\Phi}^{(m+1)}(\phi,t,\rho)}{\varepsilon^{b}},
 \end{eqnarray}
 where $\omega^{(m)}(\rho)=\frac{\partial H^{(m)}_0}{\partial I}(\rho)$ and
\begin{equation}\label{fc5-25}
R_1=\frac{1}{\varepsilon^{a}}\int_{0}^{1}(1-\tau)\frac{\partial^2 H^{(m)}_0}{\partial I^2}(\rho+\tau\frac{\partial S}{\partial \theta}) (\frac{\partial S}{\partial \theta})^{2}d\tau+\frac{1}{\varepsilon^b}\int_{0}^{1}\frac{\partial \tilde{P}^{(m)}}{\partial I}(\theta, t, \rho+\tau\frac{\partial S}{\partial\theta})\frac{\partial S}{\partial \theta}d\tau.
\end{equation}
Expanding $\tilde{P}^{(m)}(\theta,t,\rho)$ into a Fourier series,
\begin{equation}\label{fc5-26}
\tilde{P}^{(m)}(\theta,t,\rho)=\sum_{(k,l)\in\mathbf{Z}^d\times \mathbf{Z}}\widehat{\tilde{P}^{(m)}}(k,l,\rho)e^{i(\langle k,\theta\rangle+lt)}:=\tilde{P}_{1}^{(m)}(\theta,t,\rho)+\tilde{P}_{2}^{(m)}(\theta,t,\rho),
\end{equation}
where $\tilde{P}_{1}^{(m)}=\sum_{|k|+|l|\leq \tilde{K}_m}\widehat{\tilde{P}^{(m)}}(k,l,\rho)e^{i(\langle k,\theta\rangle+lt)}$, $\tilde{P}_{2}^{(m)}=\sum_{|k|+|l|> \tilde{K}_m}\widehat{\tilde{P}^{(m)}}(k,l,\rho)e^{i(\langle k,\theta\rangle+lt)}$.
Then, we derive the homological equation:
\begin{equation}\label{fc5-27}
\frac{\partial S}{\partial t}+\langle \frac{\omega^{(m)}(\rho)}{\varepsilon^{a}}, \frac{\partial S}{\partial\theta}\rangle
+\frac{ \tilde{P}_1^{(m)}(\theta,t,\rho)}{\varepsilon^{b}}-\frac{\widehat{\tilde{P}^{(m)}}(0,0,\rho)}{\varepsilon^{b}}=0,
\end{equation}
where $\widehat{\tilde{P}^{(m)}}(0, 0, \rho)$ is $0$-Fourier coefficient of $\tilde{P}^{(m)}(\theta,t,\rho)$ as the function of $(\theta,t)$. Let
\begin{equation}\label{fc5-28}
S(\theta,t,\rho)=\sum_{|k|+|l|\leq \tilde{K}_m,(k,l)\neq (0,0)}\widehat{S}(k, l, \rho)e^{i(\langle k, \theta\rangle+lt)}.
\end{equation}
By passing to Fourier coefficients, we have
\begin{equation}\label{fc5-29}
\widehat{S}(k, l, \rho)=\frac{i}{\varepsilon^b}\cdot\frac{\widehat{\tilde{P}^{(m)}}(k, l, \rho)}{\varepsilon^{-a}\langle k, \omega^{(m)}(\rho)\rangle +l},\;|k|+|l|\leq \tilde{K}_m,\; (k,l)\in \mathbf{Z}^{d}\times \mathbf{Z}\setminus\{(0,0)\}.
\end{equation}
Then we can solve homological equation (\ref{fc5-27}) by setting
\begin{equation}\label{fc5-30}
S(\theta,t,\rho)=\sum_{|k|+|l|\leq \tilde{K}_m,(k,l)\neq (0,0)}
\frac{i}{\varepsilon^b}\cdot\frac{\widehat{\tilde{P}^{(m)}}(k, l, \rho)e^{i(\langle k, \theta\rangle+lt)}}{\varepsilon^{-a}\langle k, \omega^{(m)}(\rho)\rangle +l}.
\end{equation}
By (\ref{fc1-20}), (\ref{fc5-15}) and (\ref{fc5-17}), for $\forall \rho\in \tilde{D}(\tilde{s}_m,\tilde{r}_m(\tilde{I}_m))$, $|k|+|l|\leq \tilde{K}_{m}$, $(k,l)\neq (0,0)$, we have
\begin{eqnarray}\label{fc5-31}
|\varepsilon^{-a}\langle k, \omega^{(m)}(\rho)\rangle+l|
&\geq &|\varepsilon^{-a}\langle k, \omega^{(m)}(\tilde{I}_m)\rangle+l|-|\varepsilon^{-a}\langle k, \omega^{(m)}(\tilde{I}_m)-\omega^{(m)}(\rho)\rangle|\nonumber\\
&\geq & \frac{\gamma}{|k|^{\tau_2}}-C\varepsilon^{-a}|k|\tilde{r}_{m}\nonumber\\
&\geq &\frac{\gamma}{2|k|^{\tau_2}}.
\end{eqnarray}
Then, by (\ref{fc5-16}), (\ref{fc5-29})-(\ref{fc5-31}), using R\"ussmann \cite{a27, a28} subtle
arguments to give optimal estimates of small divisor series (also see Lemma 5.1 in \cite{a29}),  we get
\begin{equation}\label{fc5-32}
\|S(\theta,t,\rho)\|_{\tilde{D}(\tilde{s}^{(1)}_m, \tilde{r}_m(\tilde{I}_m))}
\leq\frac{C\|\tilde{P}^{(m)}\|_{\tilde{D}(\tilde{s}_m,\tilde{r}_m(\tilde{I}_m))}}{\gamma\varepsilon^b\tilde{s}_{m}^{\tau_{2}}}\leq\frac{C\tilde{\varepsilon}_m}{\gamma\varepsilon^b\tilde{s}_{m}^{\tau_{2}}}.
\end{equation}
Then by the Cauchy's estimate, we have
\begin{equation}\label{fc5-33}
 \|\frac{\partial S}{\partial \theta}\|_{D(\tilde{s}^{(2)}_m, \tilde{r}_m(\tilde{I}_m))}\leq \frac{C\tilde{\varepsilon}_m}{\gamma\varepsilon^b\tilde{s}_{m}^{\tau_{2}+1}}\ll \tilde{r}_m-\tilde{r}_{m+1},\ \|\frac{\partial S}{\partial \rho}\|_{D(\tilde{s}^{(1)}_m, \tilde{r}^{(1)}_m(\tilde{I}_m))}\leq \frac{C\tilde{\varepsilon}_m}{\gamma\varepsilon^b\tilde{s}_{m}^{\tau_{2}}\tilde{r}_m}\ll \tilde{s}_m-\tilde{s}_{m+1}.
\end{equation}
By (\ref{fc5-22}) and (\ref{fc5-33}) and the implicit function theorem, we get that there are analytic functions
$u=u(\phi, t, \rho), v=v(\phi, t, \rho)$ defined on
the domain $\tilde{D}(\tilde{s}^{(3)}_m, \tilde{r}^{(3)}_m(\tilde{I}_m))$ with
\begin{equation}\label{fc5-34}
\frac{\partial S(\theta,t,\rho)}{\partial \theta}=u(\phi, t, \rho),\ \frac{\partial S(\theta,t,\rho)}{\partial \rho}=-v(\phi, t, \rho)
\end{equation}
and
\begin{equation}\label{fc5-35}
 \|u\|_{\tilde{D}(\tilde{s}^{(3)}_m, \tilde{r}^{(3)}_m(\tilde{I}_m))}\leq \frac{C\tilde{\varepsilon}_m}{\gamma\varepsilon^b\tilde{s}_{m}^{\tau_{2}+1}}\ll \tilde{r}_m-\tilde{r}_{m+1},\ \|v\|_{\tilde{D}(\tilde{s}^{(3)}_m, \tilde{r}^{(3)}_m(\tilde{I}_m))}\leq \frac{C\tilde{\varepsilon}_m}{\gamma\varepsilon^b\tilde{s}_{m}^{\tau_{2}}\tilde{r}_m}\ll \tilde{s}_m-\tilde{s}_{m+1}
\end{equation}
such that
\begin{equation}\label{fc5-36}
\tilde{\Phi}_{m+1}: \begin{cases}
              I=\rho+u(\phi, t, \rho), \\
              \theta=\phi+v(\phi, t, \rho),\\ t=t.
    \end{cases}
\end{equation}
Then, we have
\begin{equation}\label{fc5-37}
\tilde{\Phi}_{m+1}(\tilde{D}(\tilde{s}^{(3)}_m, \tilde{r}^{(3)}_m(\tilde{I}_m)))\subseteq \tilde{D}(\tilde{s}_{m},\tilde{r}_{m}(\tilde{I}_m)).
\end{equation}

Let
\begin{equation}\label{fc5-38}
H_0^{(m+1)}(\rho)=H_0^{(m)}(\rho)+\varepsilon^{a-b}\widehat{\tilde{P}^{(m)}}(0, 0, \rho).
\end{equation}
By the Cauchy's estimate and (\ref{fc5-16}), we have
 \begin{equation}\label{fc5-39}
\|\frac{\partial^p \widehat{\tilde{P}^{(m)}}(0, 0, \rho)}{\partial \rho^p}\|_{\tilde{D}(0,\tilde{r}^{(p)}_m(\tilde{I}_m))}\leq \frac{C \tilde{\varepsilon}_m}{\tilde{r}_m^p},\ \ p=1,2.
 \end{equation}
 By  (\ref{fc5-17}), (\ref{fc5-18}), (\ref{fc5-38}) and (\ref{fc5-39}),  we have
 \begin{equation}\label{fc5-40}
det\left(\frac{\partial^2 H^{(m+1)}_0(\rho)}{\partial \rho^2}\right),\ det\left(\frac{\partial^2 H^{(m+1)}_0(\rho)}{\partial \rho^2}\right)^{-1}\leq M_{m+1}, \ \forall \rho\in \tilde{D}(0,\tilde{r}^{(2)}_m(\tilde{I}_m))
 \end{equation}
 and
 \begin{equation}\label{fc5-41}
\|\frac{\partial^2 H^{(m+1
)}_0(\rho)}{\partial \rho^2}\|_{\tilde{D}(0,\tilde{r}^{(2)}_m(\tilde{I}_m))}\leq h_{m+1}.
 \end{equation}
 By  (\ref{fc5-38}), we have
 \begin{equation}\label{fc5-42}
\frac{\partial H_0^{(m+1)}(\rho)}{\partial \rho}=\frac{\partial H_0^{(m)}(\rho)}{\partial \rho}+\varepsilon^{a-b}\frac{\partial \widehat{\tilde{P}^{(m)}}(0, 0, \rho)}{\partial \rho}.
\end{equation}
Noting that $H_0^{(m+1)}(\rho)$, $H_0^{(m)}(\rho)$ and $\widehat{\tilde{P}^{(m)}}(0, 0, \rho)$ are real analytic on $\tilde{D}(0,\tilde{r}^{(2)}_m(\tilde{I}_m))$ and that $\tilde{I}_{m}\in\mathbf{R^d}$. Then by (\ref{fc5-15}), (\ref{fc5-38})-(\ref{fc5-40}) and (\ref{fc5-42}), it is not
difficult to see that (see  Appendix ``A The Classical Implicit Function Theorem'' in \cite{a31})  there exists an unique point $\tilde{I}_{m+1}\in\mathbf{R^d}$ so that
\begin{equation}\label{fc5-43}
\frac{\partial H^{(m+1)}_0}{\partial \rho}(\tilde{I}_{m+1})=\omega(I_0),
 \end{equation}
 \begin{equation}\label{fc5-44}
|\tilde{I}_{m+1}-\tilde{I}_{m}|\leq \frac{C\varepsilon^{a-b}\tilde{\varepsilon}_m}{\tilde{r}_m}\ll \tilde{r}_{m}.
 \end{equation}
  By  (\ref{fc5-37}) and (\ref{fc5-44}), we have
  \begin{eqnarray}\label{fc5-45}
\tilde{\Phi}_{m+1}(\tilde{D}(\tilde{s}_{m+1}, \tilde{r}_{m+1}(\tilde{I}_{m+1})))&\subseteq&\tilde{\Phi}_{m+1}(\tilde{D}(\tilde{s}^{(9)}_m, \tilde{r}^{(9)}_m(\tilde{I}_{m+1})))\nonumber\\
&\subseteq&\tilde{\Phi}_{m+1}(\tilde{D}(\tilde{s}^{(3)}_m, \tilde{r}^{(3)}_m(\tilde{I}_m)))\subseteq \tilde{D}(\tilde{s}_{m},\tilde{r}_{m}(\tilde{I}_m)).
\end{eqnarray}

Let
\begin{equation}\label{fc5-46}
\frac{\tilde{P}^{(m+1)}(\phi,t,\rho)}{\varepsilon^{b}}=\frac{\tilde{P}_{2}^{(m)}(\theta,t,\rho)}{\varepsilon^{b}}+\frac{P_{m_0+m+1}\circ\Phi^{(m_0)}\circ\Psi\circ\tilde{\Phi}^{(m+1)}(\phi,t,\rho)}{\varepsilon^{b}}+R_1.
\end{equation}
Then by (\ref{fc5-24}), (\ref{fc5-26}), (\ref{fc5-27}), (\ref{fc5-38}) and (\ref{fc5-46}), we have
\begin{equation}\label{fc5-47}
\tilde{H}^{(m+1)}(\phi,t,\rho)=\frac{H^{(m+1)}_0(\rho)}{\varepsilon^{a}}+ \frac{\tilde{P}^{(m+1)}(\phi,t,\rho)}{\varepsilon^{b}}+ \sum_{\nu=m_0+m+2}^{\infty}\frac{P_{\nu}\circ\Phi^{(m_0)}\circ\Psi\circ\tilde{\Phi}^{(m+1)}(\phi,t,\rho)}{\varepsilon^{b}}.
 \end{equation}
 By (\ref{fc5-16}), (\ref{fc5-26}) and (\ref{fc5-44}), it is not difficult to show that (see Lemma A.2 in \cite{a30}), we have
 \begin{equation}\label{fc5-48}
\| \frac{\tilde{P}_2^{(m)}(\theta,t,\rho)}{\varepsilon^{b}}\|_{\tilde{D}(\tilde{s}^{(9)}_m, \tilde{r}^{(9)}_m(\tilde{I}_{m+1}))}\leq \frac{C\tilde{\varepsilon}_{m}}{\varepsilon^{b}}\tilde{K}_{m}^{d+1}e^{-\frac{\tilde{K}_{m}\tilde{s}_{m}}{2}}\leq\frac{C\tilde{\varepsilon}_{m+1}}{\varepsilon^{b}}.
\end{equation}
By (\ref{fc5-16}), (\ref{fc5-18}), (\ref{fc5-32})-(\ref{fc5-35}) and (\ref{fc5-44}), we have
\begin{equation}\label{fc5-49}
\| \frac{1}{\varepsilon^{a}}\int_{0}^{1}(1-\tau)\frac{\partial^2 H^{(m)}_0}{\partial I^2}(\rho+\tau\frac{\partial S}{\partial \theta}) (\frac{\partial S}{\partial \theta})^{2}d\tau\|_{\tilde{D}(\tilde{s}^{(9)}_m, \tilde{r}^{(9)}_m(\tilde{I}_{m+1}))}\leq \frac{C}{\varepsilon^{a}}\cdot(\frac{\tilde{\varepsilon}_m}{\gamma\varepsilon^b\tilde{s}_{m}^{\tau_{2}+1}})^2\leq\frac{C\tilde{\varepsilon}_{m+1}}{\varepsilon^{b}},
\end{equation}
\begin{equation}\label{fc5-50}
\| \frac{1}{\varepsilon^b}\int_{0}^{1}\frac{\partial \tilde{P}^{(m)}}{\partial I}(\theta, t, \rho+\tau\frac{\partial S}{\partial\theta})\frac{\partial S}{\partial \theta}d\tau\|_{\tilde{D}(\tilde{s}^{(9)}_m, \tilde{r}^{(9)}_m(\tilde{I}_{m+1}))}\leq \frac{C\tilde{\varepsilon}_m}{\varepsilon^{b}\tilde{r}_m}\cdot\frac{\tilde{\varepsilon}_m}{\varepsilon^b\gamma \tilde{s}_{m}^{\tau_{2}+1}}\leq\frac{C\tilde{\varepsilon}_{m+1}}{\varepsilon^{b}}.
\end{equation}

By (\ref{fc5-35}) and (\ref{fc5-36}), we have
\begin{equation}\label{fc5-51}
\tilde{\Phi}_{m+1}(\phi,t,\rho)=(\theta,t,I),\ (\phi,t,\rho)\in \tilde{D}(\tilde{s}_m^{(3)},\tilde{r}_m^{(3)}(\tilde{I}_m)).
\end{equation}
By (\ref{fc5-35}), (\ref{fc5-36}) and (\ref{fc5-51}), we have
\begin{equation}\label{fc5-52}
\|I-\rho\|_{\tilde{D}(\tilde{s}_m^{(3)},\tilde{r}_m^{(3)}(\tilde{I}_m))}\leq \frac{C\tilde{\varepsilon}_m}{\gamma\varepsilon^b\tilde{s}_{m}^{\tau_{2}+1}}, \ \|\theta-\phi\|_{\tilde{D}(\tilde{s}_m^{(3)},\tilde{r}_m^{(3)}(\tilde{I}_m))}\leq \frac{C\tilde{\varepsilon}_m}{\gamma\varepsilon^b\tilde{s}_{m}^{\tau_{2}}\tilde{r}_m}.
\end{equation}
By (\ref{fc5-36}), (\ref{fc5-52}) and Cauchy's estimate, we have
\begin{equation}\label{fc5-53}
\|\partial(\tilde{\Phi}_{m+1}-id)\|_{\tilde{D}(\tilde{s}_m^{(4)},\tilde{r}_m^{(4)}(\tilde{I}_m))}\leq \frac{C\tilde{\varepsilon}_m}{\gamma\varepsilon^b\tilde{s}_{m}^{\tau_{2}+1}\tilde{r}_m}.
\end{equation}
By (\ref{fc5-44}) and (\ref{fc5-53}), we have
\begin{equation}\label{fc5-54}
\|\partial(\tilde{\Phi}_{m+1}-id)\|_{\tilde{D}(\tilde{s}_{m+1},\tilde{r}_{m+1}(\tilde{I}_{m+1}))}\leq \frac{1}{2^{m+2}}.
\end{equation}
By (\ref{fc5-13}), (\ref{fcg5-1}), (\ref{fc5-45}) and (\ref{fc5-54}), we have
 \begin{eqnarray}\label{fc5-55}
 \nonumber &&\|\partial\tilde{\Phi}^{(m+1)}(\phi,t,\rho)\|_{\tilde{D}(\tilde{s}_{m+1},\tilde{r}_{m+1}(\tilde{I}_{m+1}))}\\
 \nonumber&=&\|(\partial\tilde{\Phi}_1\circ\tilde{\Phi}_2\circ\cdots\circ\tilde{\Phi}_{m+1})(\partial\tilde{\Phi}_2\circ\tilde{\Phi}_3\circ\cdots\circ\tilde{\Phi}_{m+1})\cdots(\partial\tilde{\Phi}_{m+1})\|_{\tilde{D}(\tilde{s}_{m+1},\tilde{r}_{m+1}(\tilde{I}_{m+1}))}\\
   \nonumber&\leq&\prod_{j=0}^{m}(1+\frac{1}{2^{j+2}})\\
   &\leq&2.
\end{eqnarray}
 We claim  that \begin{equation}\label{fc5-56}
\|P_{\nu}\circ\Phi^{(m_0)}\circ\Psi\circ\tilde{\Phi}^{(m+1)}(\phi,t,\rho)\|_{\tilde{D}(\tilde{s}_{\nu-m_0},\tilde{r}_{\nu-m_0}(\tilde{I}_{m+1}))}\leq  C\tilde{\varepsilon}_{\nu-m_0},\ \nu=m_0+m+1,m_0+m+2,\cdots.
 \end{equation}
 In fact,  suppose that $w=\tilde{\Phi}^{(m+1)}(z)$ with $z=(\phi,t,\rho)\in \tilde{D}(\tilde{s}_{\nu-m_0},\tilde{r}_{\nu-m_0}(\tilde{I}_{m+1}))$.
Since $\tilde{\Phi}^{(m+1)}$ is real for real argument and $\tilde{r}_{\nu-m_0}<\tilde{s}_{\nu-m_0}$,  we have
 \begin{eqnarray}\label{fc5-57}
 \nonumber &&|{\rm Im} w|=|{\rm Im} \tilde{\Phi}^{(m+1)}(z)|=|{\rm Im} \tilde{\Phi}^{(m+1)}(z)-{\rm Im} \tilde{\Phi}^{(m+1)}({\rm Re}z)|\\
 \nonumber&\leq&| \tilde{\Phi}^{(m+1)}(z)- \tilde{\Phi}^{(m+1)}({\rm Re}z)|\\
   \nonumber&\leq&\|\partial\tilde{\Phi}^{(m+1)}(\phi,t,\rho)\|_{\tilde{D}(\tilde{s}_{m+1},\tilde{r}_{m+1}(\tilde{I}_{m+1}))}|{\rm Im} z|\\
   &\leq&2|{\rm Im} z|\leq2\tilde{s}_{\nu-m_0}.
\end{eqnarray}
By (\ref{fc5-13}), (\ref{fc5-45}) and (\ref{fc5-57}), we have
\begin{equation}\label{fc5-58}
\tilde{\Phi}^{(m+1)}(\tilde{D}(\tilde{s}_{\nu-m_0},\tilde{r}_{\nu-m_0}(\tilde{I}_{m+1})))\subseteq D_{\nu}:=(\mathbf{T}^{d+1}_{2\tilde{s}_{\nu-m_0}}\times \mathbf{R}^{d}_{2\tilde{s}_{\nu-m_0}})\bigcap \tilde{D}(\tilde{s}_{0},\tilde{r}_{0}(\tilde{I}_{0})).
\end{equation}
For $(t,I)=(t_1+t_2i,I_1,I_2i)\in D_{\nu}$, where $t_1,t_2,I_1,I_2$ are real numbers,  we have
\begin{eqnarray}\label{fc5-59}
 \nonumber&&\|{\rm Im}\frac{\partial \tilde S(t,I)}{\partial I}\|_{D_{\nu}}\\
 \nonumber &=&\|{\rm Im}\frac{\partial \tilde S}{\partial I}(t_1+t_2i,I_1+I_2i)-{\rm Im}\frac{\partial \tilde S}{\partial I}(t_1,I_1)\|_{D_{\nu}}\\
\nonumber&\leq&\|\frac{\partial \tilde S}{\partial I}(t_1+t_2i,I_1+I_2i)-\frac{\partial \tilde S}{\partial I}(t_1,I_1)\|_{D_{\nu}}\\
\nonumber&\leq&\|\frac{\partial^2 \tilde S(t,I)}{\partial I \partial t}\|_{\mathcal{D}}\|t_2i\|_{D_{\nu}}+\|\frac{\partial^2 \tilde S(t,I)}{\partial^2 I}\|_{\mathcal{D}}\|I_2i\|_{D_{\nu}}\\
\nonumber&\leq&\frac{C\tilde{s}_{\nu-m_0}}{\varepsilon^br_{m_0}s_{m_0}}+\frac{C\tilde{s}_{\nu-m_0}}{\varepsilon^br^2_{m_0}}\\
&\leq&\frac{1}{2}s_{\nu}.
\end{eqnarray}
By (\ref{fbc5-1}), (\ref{fc5-58}) and (\ref{fc5-59}), we have
\begin{equation}\label{fc5-60}
\Psi\circ\tilde{\Phi}^{(m+1)}(\tilde{D}(\tilde{s}_{\nu-m_0},\tilde{r}_{\nu-m_0}(\tilde{I}_{m+1})))\subseteq \bar{D}_{\nu}:=(\mathbf{T}^{d+1}_{s_{\nu}}\times \mathbf{R}^{d}_{2\tilde{s}_{\nu-m_0}})\bigcap D(s_{m_0},r_{m_0}).
\end{equation}
Suppose that $w=\Phi^{(m_0)}(z)$ with $z=(\theta,t,I)\in \bar{D}_{\nu}$.
Since $\Phi^{(m_0)}$ is real for real argument  and $2\tilde{s}_{\nu-m_0}<r_{\nu}<s_{\nu}$, then by (\ref{fc3-44}) with $m=m_0-1$, we have
 \begin{eqnarray}\label{fc5-62}
 \nonumber &&|{\rm Im} w|=|{\rm Im} \Phi^{(m_0)}(z)|=|{\rm Im} \Phi^{(m_0)}(z)-{\rm Im} \Phi^{(m_0)}({\rm Re}z)|\\
 \nonumber&\leq&| \Phi^{(m_0)}(z)- \Phi^{(m_0)}({\rm Re}z)|\\
   \nonumber&\leq&\|\partial\Phi^{(m_0)}(\theta,t,I)\|_{D(s_{m_0},r_{m_0})}|{\rm Im} z|\\
   &\leq&2|{\rm Im} z|\leq2s_{\nu}.
\end{eqnarray}
Then by (\ref{fc5-60}) and (\ref{fc5-62}), we have
\begin{equation}\label{fc5-63}
\Phi^{(m_0)}\circ\Psi\circ\tilde{\Phi}^{(m+1)}(\tilde{D}(\tilde{s}_{\nu-m_0},\tilde{r}_{\nu-m_0}(\tilde{I}_{m+1})))\subset \mathbf{T}^{d+1}_{2s_{\nu}}\times \mathbf{R}^{d}_{2s_{\nu}},\ \nu=m_0+m+1,m_0+m+2,\cdots.
\end{equation}
By (\ref{fc3-3}) and (\ref{fc5-63}), the proof of (\ref{fc5-56}) is completed. By (\ref{fc5-25}), (\ref{fc5-44}), (\ref{fc5-46}), (\ref{fc5-48})-(\ref{fc5-50}) and (\ref{fc5-56}), we have
 \begin{equation}\label{fc5-64}
\|\tilde{P}^{(m+1)}\|_{\tilde{D}(\tilde{s}_{m+1},\tilde{r}_{m+1}(\tilde{I}_{m+1}))}\leq C\tilde{\varepsilon}_{m+1}.
\end{equation}
Then the proof  is completed by  (\ref{fc5-40}), (\ref{fc5-41}), (\ref{fc5-43}), (\ref{fc5-45}), (\ref{fc5-47}), (\ref{fc5-54}), (\ref{fc5-56})  and (\ref{fc5-64}).
 \end{proof}
 \section{Proof of  Theorems \ref{thm1-1}-\ref{thm1-2}}\label{sec6}
 In Lemma \ref{lem5-1}, letting $m\rightarrow\infty$ we get the following lemma:
 \begin{lem}\label{lem6-1}  There exisits a symplectic transformation $\tilde{\Phi}^{(\infty)}:=\lim_{m\rightarrow\infty}\tilde{\Phi}_0\circ\tilde{\Phi}_1\circ\cdots\circ\tilde{\Phi}_{m}$ with
\begin{equation}\label{fc6-1}
\tilde{\Phi}^{(\infty)}:\mathbf{T}^{d+1}\times \{\tilde{I}_{\infty}\}\rightarrow D(\tilde{s}_0,\tilde{r}_0(\tilde{I}_0)),
 \end{equation}
where $\tilde{I}_{\infty}\in \mathbf{R^{d}}$ such that system {\rm(\ref{fc5-3})} is changed by $\tilde{\Phi}^{(\infty)}$ into
\begin{equation}\label{fc6-2}
\tilde{H}^{(\infty)}(\theta,t,I)=\tilde{H}^{(0)}\circ\tilde{\Phi}^{(\infty)}=\frac{H^{(\infty)}_0(I)}{\varepsilon^{a}},
 \end{equation}
where
 \begin{equation}\label{fc6-3}
\frac{\partial H^{(\infty)}_0}{\partial I}(\tilde{I}_{\infty})=\omega(I_0),
\end{equation}
 \begin{equation}\label{fc6-4}
\|\tilde{\Phi}^{(\infty)}-id\|_{\mathbf{T}^{d+1}\times \tilde{I}_{\infty}}\leq \tilde{\varepsilon}_0^{\frac{1}{2\ell}}.
\end{equation}
  \end{lem}
  \begin{proof} By (\ref{fc5-35}) and (\ref{fc5-55}), for $z=(\theta,t,I)\in \mathbf{T}^{d+1}\times \tilde{I}_{\infty}$ and $m=0,1,2, \cdots$, we have
 \begin{eqnarray}\label{fc6-5}
 \nonumber &&\|\tilde{\Phi}^{(m+1)}(z)-\tilde{\Phi}^{(m)}(z)\|_{\mathbf{T}^{d+1}\times \tilde{I}_{\infty}}\\
 \nonumber&=&\|\tilde{\Phi}^{(m)}(\tilde{\Phi}_{m+1}(z))-\tilde{\Phi}^{(m)}(z)\|_{\mathbf{T}^{d+1}\times \tilde{I}_{\infty}}\\
   \nonumber&\leq&\|\partial\tilde{\Phi}^{(m)}(\tilde{\Phi}_{m+1}(z))\|_{\mathbf{T}^{d+1}\times \tilde{I}_{\infty}}\|\tilde{\Phi}_{m+1}(z)-z\|_{\mathbf{T}^{d+1}\times \tilde{I}_{\infty}}\\
&\leq&2\tilde{\varepsilon}_m^{\frac{1}{\ell}},
\end{eqnarray}
where $\tilde{\Phi}^{(0)}:=id$. Then, we have
 \begin{equation*}
\|\tilde{\Phi}^{(\infty)}(z)-z\|_{\mathbf{T}^{d+1}\times \tilde{I}_{\infty}}\leq\sum_{m=0}^{\infty} \|\tilde{\Phi}^{(m+1)}(z)-\tilde{\Phi}^{(m)}(z)\|_{\mathbf{T}^{d+1}\times \tilde{I}_{\infty}}\leq\sum_{m=0}^{\infty}2\tilde{\varepsilon}_m^{\frac{1}{\ell}} \leq \tilde{\varepsilon}_0^{\frac{1}{2\ell}}.
\end{equation*}
This completes the proof of Lemma \ref{lem6-1}.
    \end{proof}
Then the proof of Theorem \ref{thm1-1} is completed by (\ref{fc3-1}), (\ref{fc3-5}), (\ref{fc3-50}), (\ref{fc4-4}) , (\ref{fc5-3}) and Lemma \ref{lem6-1}. Applying Theorem \ref{thm1-1} to (\ref{fc1-7}) we have Theorem \ref{thm1-2} (see Section 5 of \cite{a16} for the proof).
\bibliographystyle{abbrv} 
\bibliography{kam}

\vspace{6pt}



\begin{thebibliography}{aa}
\bibitem[1]{a1}
J. Moser, On invariant curves of area-preserving mappings of an annulus, {\it Nachr. Akad. Wiss. Gottingen
Math. Phys.}, {\bf 2} (1962), 1-20.
\bibitem[2]{a2}
J. Moser, Stable and Random Motions in Dynamical Systems, Ann. of Math. Studies, Princeton Uni. Press,
Princeton, NJ, 1973.
\bibitem[3]{a3}
J. E. Littlewood, Some problems in real and complex analysis, Heath, Lexington, Mass. 1968.
\bibitem[4]{a4}
G. R. Morris, A case of boundedness in Littlewood's problem on oscillatory differential equations, {\it Bull.
Austral. Math. Soc.},  {\bf 14(1)} (1976), 71-93.
\bibitem[5]{a5}
R. Dieckerhoff and E. Zehnder, Boundedness of solutions via the twist-theorem, {\it Ann. Sc. Norm.Super. Pisa}, {\bf 14(1)} (1987), 79-95.
\bibitem[6]{a6}
S. Laederich and M. Levi, Invariant curves and time-dependent potentials, {\it Ergodic. Theory Dynam. Systems}, {\bf 11(2)} (1991), 365-378.
    \bibitem[7]{a7}
     B. Liu, Boundedness for solutions of nonlinear Hill's equations with periodic forcing terms via Moser's twist theorem, {\it J. Differential Equations}, {\bf 79(2)} (1989), 304-315.

\bibitem[8]{a8}
 B. Liu, Boundedness for solutions of nonlinear periodic differential equations via Moser's twist theorem, {\it Acta Math. Sinica (N.S.)}, {\bf 8(1)} (1992), 91-98.
 \bibitem[9]{a9}
 Y. Wang, Unboundedness in a Duffing equation with polynomial potentials, {\it J. Differential Equations}, {\bf 160(2)} (2000), 467-479.

 \bibitem[10]{a10}
 L. Jiao, D. Piao and Y. Wang, Boundedness for the general semilinear Duffing equations via the twist theorem, {\it J. Differential Equations}, {\bf252(1)} (2012), 91-113.

 \bibitem[11]{a11}
 Y. Peng, D. Piao and Y. Wang, Longtime closeness estimates for bounded and unbounded solutions of non-recurrent Duffing equations with polynomial potentials, {\it J. Differential Equations}, {\bf268(2)} (2020), 513-540.

    \bibitem[12]{a12}
     X. Yuan, Invariant tori of Duffing-type equations, {\it Advances in Math. (China)}, {\bf 24} (1995), 375-376.
    \bibitem[13]{a13}
    X. Yuan, Invariant tori of Duffing-type equations, {\it J. Differential Equations}, {\bf 142(2)} (1998), 231-262.

 \bibitem[14]{a14}
  X. Yuan, Lagrange stability for Duffing-type equations, {\it J. Differential Equations}, {\bf 160(1)} (2000), 94-117.
\bibitem[15]{a15}
 X. Yuan, Boundedness of solutions for Duffing equation
with low regularity in time, {\it Chinese Annals of Mathematics, Series B}, {\bf38(5)} (2017), 1037-1046.
\bibitem[16]{a16} X. Yuan, L. Chen and J. Li,
               KAM theorem with large perturbation and application to network of Duffing
oscillators, arXiv:2104.05898v1 [math.DS] 13 Apr 2021.
\bibitem[17]{a17}  R. Jothimurugan, K. Thamilmaran, S. Rajasekar and M. A. F. Sanju\'{a}n,
                           Multiple resonance and anti-resonance in coupled Duffing oscillators,
                           {\it Nonlinear Dyn.}, {\bf 83(4)} (2016), 1803-1814.
 \bibitem[18]{a18}     I. Z. Kiss, Y. Zhai and J. L. Hudson, Resonance clustering in globally coupled electrochemical oscillators with external forcing, {\it Phy. Rev. E}, {\bf 77(4)} (2008), 046204.
\bibitem[19]{a19}     A. Kovaleva, Capture into resonance of coupled Duffing oscillators, {\it Phy. Rev. E}, {\bf 92(2)} (2015), 022909.
\bibitem[20]{a20} M. G. Clerc, S. Coulibaly, M. A. Ferr\'{e} and R. G. Rojas,
                    Chimera states in a Duffing oscillators chain coupled to nearest neighbors,
                    {\it Chaos}, {\bf 28(8)} (2018), 083126.
\bibitem[21]{a21} P. Sarkar and D. S. Ray,
                 Vibrational antiresonance in nonlinear coupled systems,
                 {\it Phy. Rev. E}, {\bf 99(5)} (2019), 052221.
\bibitem[22]{a22} A. K. Chatterjee, A. Kundu and M. Kulkarni,
                    Spatiotemporal spread of perturbations in a driven dissipative Duffing chain: An out-of-time-ordered correlator approach,
                    {\it Phy. Rev. E}, {\bf 102(5)} (2020), 052103.
\bibitem[23]{a23} J. Shena, N. Lazarides and J. Hizanidis,
                Multi-branched resonances, chaos through quasiperiodicity, and asymmetric states in a superconducting dimer,
                {\it Chaos}, {\bf 30(12)} (2020), 123127.
\bibitem[24]{a24} J. P. Deka, A. K. Sarma, A. Govindarajan and M. Kulkarni,
               Multifaceted nonlinear dynamics in PT-symmetric coupled Li\'{e}nard oscillators,
               {\it Nonlinear Dyn.}, {\bf 100(2)} (2020), 1629-1640.
\bibitem[25]{a25} D. A. Salamon, The Kolmogorov-Arnold-Moser theorem. {\it Math. Phys. Electron. J}, {\bf10(3)} (2004), 1-37.
\bibitem[26]{a26} E. Zehnder, Generalized implicit function theorems with applications to some small divisor
problems I and II, {\it Comm. Pure Appl. Math.} {\bf 28} (1975), 91-140; {\bf 29(1)} (1976), 49-111.
\bibitem[27]{a27} H. R\"ussmann, On optimal estimates for the solutions of linear partial differential equations of
first order with constant coefficients on the torus. Dynamical systems, theory and applications.
Springer, 1975, pp 598-624.
\bibitem[28]{a28} H. R\"ussmann, On the existence of invariant curves of twist mappings of an annulus. Geometric
dynamics. Springer, 1983, pp 677-718.
\bibitem[29]{a29}
J. Li, J. Qi and X. Yuan, KAM theorem for reversible mapping of low smoothness
with application, arXiv:1910.08214v1 [math.DS] 18 Oct 2019.
\bibitem[30]{a30} J. P\"oschel, A lecture on the classical KAM theorem, {\it Proc. Symp. Pure Math.}, {\bf 69} (2001)
707-732.
\bibitem[31]{a31} L. Chierchia, Kolmogorov-Arnold-Moser (KAM) Theory. Mathematics of complexity and dynamical systems. Vols. 1-3, 810-836, Springer, New York, 2012.
\end{thebibliography}

\end{document}